\documentclass[12pt]{amsart}
\usepackage[backref=page]{hyperref}
\usepackage[alphabetic,initials]{amsrefs} 
\usepackage{bm,amsmath,amssymb,amsthm}

\usepackage[all]{xy}
\usepackage{stmaryrd}

\usepackage{amsfonts}

\usepackage[utf8]{inputenc}

\usepackage{enumerate}

\theoremstyle{plain}
\newtheorem{theorem}{Theorem}

\newtheorem{corollary}{Corollary}[theorem]

\newtheorem{proposition}{Proposition}[section]
\newtheorem*{proposition*}{Proposition}
\newtheorem{lemma}[proposition]{Lemma}
\newtheorem*{lemma*}{Lemma}
\newtheorem*{corollary*}{Corollary}
\theoremstyle{definition}

\newtheorem*{remark}{Remark}
\newtheorem*{remarks}{Remarks}

\theoremstyle{remark}

\numberwithin{equation}{section}

\newcommand{\qtext}[1]{\quad\text{#1}}

\newcommand{\abs}[1]{\ensuremath{\left|#1\right|}}

\DeclareMathOperator {\Msym} {sym}
\newcommand{\Mdemi}{\frac{1}{2}}
\newcommand{\BmR}{\mathbb{R}}
\newcommand{\Mdede}[4]{
\begin{pmatrix}
#1&#2  \\
#3&#4 \\
\end{pmatrix}
}
\newcommand{\SB}{\backslash}
\newcommand{\BmZ}{\mathbb{Z}}
\newcommand{\Mpartialsquare}[1]{\frac{\partial^2}{\partial #1^2}}
\newcommand{\FmH}{\mathfrak{H}}
\newcommand{\FmD}{\mathfrak{D}}
\newcommand{\Fms}{\mathfrak{s}}
\newcommand{\Fml}{\mathfrak{l}}
\newcommand {\pnorm}[1]   {\left\lVert #1 \right\rVert}
\DeclareMathOperator {\Ccinf}  {\CmC^\infty_c}

\DeclareMathOperator {\GL} {GL}
\DeclareMathOperator {\SL} {SL}
\newcommand{\CmC}{\mathcal{C}}
\newcommand{\Mpartial}[1]{\frac{\partial}{\partial #1}}
\DeclareMathOperator {\Msgn}  {sgn}
\newcommand{\BmN}{\mathbb{N}}

\newcommand{\CmH}{\mathcal{H}}
\DeclareMathOperator {\Mch}  {ch}
\DeclareMathOperator {\MRe}  {\Re e}
\DeclareMathOperator {\MIm}  {Im}
\newcommand{\BmC}{\mathbb{C}}
\newcommand{\Z}{\mathbb{Z}}
\newcommand{\CmS}{\mathcal{S}}

\newcommand{\Lquote}[1]{``#1''}
\newcommand{\lcusp}{L^2_{\text{cusp}}(\Gamma \SB G,\chi)}
\newcommand{\lgen}{L^2(\widetilde\Gamma\SB \widetilde G,\chi)}
\newcommand{\fpart}[1]{\{#1\}}
\newcommand{\pluscont}{+\text{ cont.}}
\newcommand{\Ioff}{I_\mathrm{off}}
\newcommand{\Ires}{I_\mathrm{res}}
\newcommand{\Ht}{\mathrm{Ht}}
\newcommand{\CmO}{\mathcal{O}}

\DeclareMathOperator{\MCl}{Cl}
\newcommand{\clchar}{\widehat{\MCl}_D}
\newcommand{\BmQ}{\mathbb{Q}}
\newcommand{\Fma}{\mathfrak{a}}
\newcommand{\TmN}{\mathbf{N}} 
 
\begin{document}

\title[non-split sums]{Non-split sums of coefficients of $GL(2)$-automorphic forms}

%-----------Information for first author------
\author{Nicolas Templier}
\address{Department of Mathematics, Fine Hall, Washington Road, Princeton, NJ 08544-1000.}
\email{templier@math.princeton.edu}
\date{\today}

%---------Information for second author-------
\author{Jacob Tsimerman}
\address{Department of Mathematics, Fine Hall, Washington Road, Princeton, NJ 08544-1000.}
\email{jtsimerm@math.princeton.edu}

%---------General info--------------------
%\subjclass[2000]{Primary 54C40, 14E20; Secondary 46E25, 20C20}
%\date{July, 2007}
%\date{January 1, 2001 and, in revised form, June 22, 2001.}
%\keywords{Differential geometry, algebraic geometry}

%----------Abstract----------------------
\begin{abstract}

Given a cuspidal automorphic form $\pi$ on $\GL_2$, we study smoothed sums of the form
$\sum_{n\in\mathbb{N}} a_{\pi}(n^2+d)W(\frac{n}{Y})$. The error term we get is sharp in that it is uniform in both $d$ and $Y$
and depends directly on bounds towards Ramanujan for forms of half-integral weight and Selberg eigenvalue conjecture.
Moreover, we identify (at least in the case where the level is square-free) the main term as a simple factor times the
residue as $s=1$ of the symmetric square L-function $L(s,\Msym^2\pi)$. In particular there is no main term unless $d>0$ and $\pi$ is a dihedral form. 

\end{abstract}
%-----------------------------------------

\maketitle

\tableofcontents

\section{Introduction and statement of results}\label{sec:intro} Understanding averages of arithmetic functions over sequences that are sparse is a subject that has
attracted much historical interest. In this paper, we look at sums of Hecke eigenvalues of $\GL_2$-representations $\pi$ over quadratic
progressions $n\mapsto n^2+d$. That is, sums of the type:
\begin{equation}\label{intro:aim}
\displaystyle\sum_{n\leq X} a_{\pi}(n^2 + d),
\end{equation}
for various $d\in \BmZ$ and $X>0$.
Here $a_\pi$ stands for the normalized Fourier coefficients of $\pi$.
Because of the motivation and analogy with the problem of estimating shifted convolution sums (when the polynomial $n^2+d$ is split it reduces to shifted convolution sums) we shall call sums such as~\eqref{intro:aim} \Lquote{a non-split sum}.

Sums over polynomial progressions of the type~\eqref{intro:aim}  were first examined by Erd\"os for the divisor function $\tau(n) =
\displaystyle\sum_{m|n} 1$. For the divisor function along a fixed quadratic progression, it has been known%
\footnote{This is for instance mentioned without proof in the introduction of Bellman~\cite{Bellman}, see also Exercise~3 in~\cite{book:IK04}*{\S1.5}.}
that an asymptotic formula of the form
\begin{equation}\label{dirichlet}
\sum_{n\le X}
\tau(n^2+1)
\sim
\frac{3}{\pi} X\log X,
\qtext{as $X\to \infty$}
\end{equation}
may be derived with the Dirichlet hyperbola principle.

Hooley~\cite{Hool63} was the first to obtain a power saving in the remainder term of~\eqref{dirichlet}. The exponent was later improved by Bykovskii~\cite{Bykovskii} and Deshouillers-Iwaniec~\cite{DeI:82}. Sarnak~\cite{Sarn84} gave an interpretation of those results in the context of automorphic forms, see~\S\ref{sec:intro:works} below.

Recently, Blomer~\cite{Blom08} investigated~\eqref{intro:aim} for a fixed holomorphic modular form $f$ and a fixed $d$ and showed the following asymptotic evaluation
\begin{equation}\label{intro:blomer}
\displaystyle\sum_{n\leq X} a_f(n^2+ d) = c_f
X + O_{f,d}(X^\frac{6}{7}).
\end{equation}

Independently, the first named author was led to similar sums in~\cite{Temp:non-split} when $X$ is about $d^\frac{1}{2}$ and $d>0$. In that case the question is intimately related to the equidistribution of Heegner points. A precise estimate for those shorter sums has been derived in~\cite{Temp:non-split}. Applications to moments of $L$-functions are described in~\S\ref{sec:intro:app} below.

Our main purpose in this article is to improve and extend the above results in several aspects. Our main result is Theorem~\ref{th:main} below. We shall establish uniformity in both the $d$ and $X$ aspects simultaneously, thereby unifying the results in~\cite{Blom08} and~\cite{Temp:non-split}. 

We treat both the $d>0$ and $d<0$ cases and have removed several technical conditions from previous papers, and reached a reasonably good level of generality (the representation $\pi$ has arbitrary infinity type, there is no condition on the level). This requires adapting some advanced techniques from the shifted convolution problem as we shall discuss below. Also we shall analyse further when the main term can occur (in many cases $c_f=0$), and connect to the possible pole of $L(s,\Msym^2 f)$ and the theory of the Shimura integral. 

The quality of our remainder term is sharp, in that it depends directly on the bound towards Selberg eigenvalue conjecture (which we denote by $0\le \theta \le 1/2$) and a subconvex exponent for the Fourier coefficients of half-integral automorphic forms (which we denote by $0\le \delta\le 1/4$). See~\S\ref{sec:intro:sub} for the precise definitions and the current numerical records. It is interesting to note that these exponents are of a rather different nature. The Selberg conjecture $\theta=0$ would follow from Langlands functoriality conjecture, while $\delta=0$ is equivalent to a Lindel\"of hypothesis in the level aspect and would follow from the GRH.

\begin{theorem}\label{th:main}
 Let $\pi$ be a $\GL_2$ automorphic cuspidal
representation and denote by $a_{\pi}(n)$ the Dirichlet coefficients of the $L$-function $L(s,\pi)$. Let $W$ be a smooth function of compact support inside $(1,2)$ with $W^{(i)}\ll 1$ for all $i\ge 0$. Then for all $\epsilon>0$,
\begin{equation}\label{eq:th:main}
 \sum_{n\ge 0} a_\pi(n^2+d) W(\frac{n^2+d}{Y}) = 
 I(W)M_{\pi,d} \sqrt{Y}
  + O_{\pi,\epsilon}(Y^{1/4+\epsilon} \abs{d}^{\delta} (\frac{Y}{\abs{d}})^{\theta/2}),
\end{equation}
uniformly, for all integers $d\neq 0$ and reals $Y\gg \abs{d}$. The implied multiplicative constant in the remainder term depends only on $\pi$ and $\epsilon>0$. The multiplicative constant $M_{\pi,d}$ is non-zero only if $\pi$ is dihedral and $d>0$.
\end{theorem}
\remarks \begin{enumerate}[(i)]

\item  When $d>0$ we may suppose $Y\ge d/2$ otherwise the sum in the left-hand side of~\eqref{eq:th:main} is zero.

\item This generalizes the result by Blomer~\cite{Blom08} recalled above which corresponds to the case when $\pi$ is holomorphic and $d$ is fixed. With some extra effort it should be possible, as in~\cite{Blom08}, to replace the sequence $n\mapsto n^2+d$ by the sequence $n\mapsto f(n)$ for a quadratic polynomial $f$. Blomer noted that the constant $M_{\pi,d}$ is non-zero only when the weight of $\pi$ is odd and $d>0$. This is a special case of our result because dihedral holomorphic forms have odd weight.

\item Under the Lindel\"of hypothesis and the Selberg eigenvalue conjecture we would have $\delta=\theta=0$, see~\S\ref{sec:intro:sub} below. The error term would then be $Y^{1/4+\epsilon}$ which is roughly the square-root of the length of summation.

\item When $d<0$ the assumption $Y\gg \abs{d}$ is slightly restrictive. We believe it would be possible to relax the assumption slightly, at the cost of a worse remainder term. A similar phenomenon in the split shifted convolution problem occurs in that range as well, see the discussion in~\S\ref{sec:intro:works} below. 

It is also interesting to discuss the significance of the assumption $Y\gg \abs{d}$ in analysing briefly the situation when $Y$ is significantly smaller than $\abs{d}$. The first observation is that this is rather subtle. If $\frac{Y}{\abs{d}}$ becomes significantly smaller than one, the exponent $\theta/2$ in the remainder term of~\eqref{eq:th:main} would act in the wrong direction. Note also that the length of summation in~\eqref{eq:th:main} becomes roughly $\frac{Y}{\abs{d}^{1/2}}$ instead of $\sqrt{Y}$ when $Y$ is large. The assumption $Y\gg \abs{d}$ occurs naturally in the proof because the multiplicative factor $e^{-d/Y}$ needs to remain bounded. A further subtlety concerns the asymptotics of Whittaker function, if $Y$ were significantly smaller than $\abs{d}$ we would enter in a different regime.
\end{enumerate}

For the sake of clarity we would like to have an estimate that is a direct analogue of~\eqref{dirichlet} and~\eqref{intro:blomer}. The following is a formal consequence of Theorem~\ref{th:main}.
\begin{corollary}\label{cor:main:X} Let $\pi$ be a $\GL_2$ automorphic cuspidal
representation with trivial central character. Let $V$ be a smooth function of compact support inside $(1,2)$ with $V^{(i)}\ll 1$ for all $i\ge 0$. For all integer $d>0$ and all $X\gg \abs{d}^{1/2}$,
\begin{equation}\label{eq:cor:main:X}
\sum_{n\ge 1} a_\pi(n^2+d) V(\frac{n}{X}) \ll_{\pi,\epsilon} X^{1/2+\epsilon}
d^\delta (1+\frac{X^2}{d})^{\theta/2}.
\end{equation}
\end{corollary}
\remarks \begin{enumerate}[(i)]
         \item When $\pi$ is holomorphic, the trivial bound obtained from the triangle inequality is $X^{1+\epsilon}$ because of Deligne's bound. When $\pi$ is non-holomorphic, the trivial upper bound would be $X^{1+\theta}$. In fact the estimate~\eqref{eq:cor:main:X} also gains on the bound $X^{1+\epsilon}$ when $\pi$ is Maass, so we always view $X^{1+\epsilon}$ as \Lquote{the trivial bound} in the discussions below. The estimate~\eqref{eq:cor:main:X} exhibits cancellation compared to the trivial bound when $X$ is large enough. For instance it exhibits cancellations when $X>d^{1/2+\epsilon}$, which is a natural barrier, and the natural range of uniformity required for applications.
        
\item  There is no main term because $M_{\pi,d}=0$ from the fact that $\pi$ cannot be dihedral since its central character is trivial. It is not difficult to derive as well an asymptotic in the general case of non-trivial central character and $d<0$, in which case there might be a main term when $M_{\pi,d}\neq 0$.

\item In the region when $X$ is much larger than $d$, the bound becomes $X^{1/2+\theta}$. For instance when $d$ is fixed we achieve the same numerical exponent as in Blomer~\cite{Blom08}*{Theorem~2}.

\item  Assuming the Selberg eigenvalue conjecture ($\theta=0$) and the Lindel\"of hypothesis in the level aspect ($\delta=0$), the upper bound in~\eqref{eq:cor:main:X} would be $X^{1/2+\epsilon}$.

\item One can go beyond $X\gg \abs{d}^{1/2}$ by being more flexible with the function $W$ in the Theorem~\ref{th:main}. In particular one can make the error term in~\eqref{eq:th:main} depend on a sufficiently large Sobolev norm $\pnorm{W}_A$ as in~\cite{BH08}*{\S2.3}. Then one can go down to $X\gg \abs{d}^{1/2-\eta}$ where $\eta>0$ is inversely proportional to $A$. In that way we would recover the main result of~\cite{Temp:non-split}. 
\end{enumerate}

Because of our soft treatment (that is, we avoid using complicated transforms by exploiting the framework of representation theory), we get the uniformity in both the $d$ and $X$ aspects simultaneously (or $Y$ in the context of Theorem~\ref{th:main}). We continue this introduction with an important application to moments of $L$-functions, some discussions on the bounds $\theta$ and $\delta$, a comparison to the split shifted convolution problem, and a detailed outline of proofs and the structure of the article. 

\subsection{Square-free level and holomorphic forms}

We record here the exact form of $I(W)$ and $M_{\pi,d}$ in the case where $\pi$ is a discrete series representation corresponding to a holomorphic form
of weight $K$, with square-free level.
\begin{theorem}
 Let $\pi$ be a $\GL_2$ automorphic cuspidal
representation 	corresponding to a holomorphic form of weight $K$, nebentypus $\chi$ and square-free level $N$, and denote by $a_{\pi}(n)$ denote the Dirichlet coefficients of the $L$-function $L(s,\pi)$. Let $W$ be a smooth function of compact support inside $(1,2)$ with $W^{(i)}\ll 1$ for all $i\ge 0$. Then for all $\epsilon>0$,
\begin{equation}\label{eq:th:hol}
 \sum_{n\ge 0} a_\pi(n^2+d) W(\frac{n^2+d}{Y}) = 
 I(W)M_{\pi,d} \sqrt{Y}
  + O_{\pi,\epsilon}(Y^{1/4+\epsilon} \abs{d}^{\delta} (\frac{Y}{\abs{d}})^{\theta/2}),
\end{equation}

Where $M_{\pi,d}$ is 0 unless $d$ is $N$ times a perfect square, $\chi=\chi_{4N}$ and $K\notin2\Z$, in which case
\[M_{\pi,d}= 
\frac{2^{K/2}(4\pi)^{1/4}\zeta^{-1}(2)}{\Gamma(3/4-k/2)}Res_{s=1}L(s,\Msym^2\pi)
\displaystyle\prod_{i=1}^{\frac{K-1}{2}}\Bigl(i(i-1/2)\Bigr)^{-1}\] and \[I(W)=\int_{0}^{\infty}W(y)y^{\frac12}\frac{dy}{y}.\]
\end{theorem}

\begin{remark} As the proof given in \S\ref{sec:hol} shows, given a newform $\phi$, the main term $M_{\pi,d}$ can always be expressed, up to some elementary
factors depending on the infinity type, nebentypus $\chi$, level $N$ and weight $K$ of $\pi$, as the inner product of $\phi\overline\theta(z)$ with a theta function 
$\theta_{K,N,\chi}(z)$ of half-integral weight in the residual spectrum depending only on the weight, level, and nebentypus of $\pi$. The latter restriction on $d$ being $N$ times a perfect square comes from a classical theorem of Serre and Stark, which says that the residual spectrum is spanned purely by theta functions.
\end{remark}

\subsection{Applications}\label{sec:intro:app} 
The Theorem~\ref{th:main} arises in the study of moments of $L$-functions associated to quadratic number fields. We recall that moments of $L$-functions in families are a central tool to the problem of non-vanishing and subconvexity for special values.

Let $D<0$ be the discriminant of an imaginary quadratic field $K=\BmQ(\sqrt{D})$. Let $\CmO_D$ be the ring of integers and $\MCl_D$ be the ideal class group. We let $h(D)=\abs{\MCl_D}$ be the order of the class group (class number). To unitary characters $\chi\in \clchar$ one may associate interesting $L$-functions. For each discriminant $D$, we consider the $L$-functions associated to the various $\chi\in \clchar$ together and then form various averages. For a list and comparison of those possible $L$-functions, see~\cite{Temp:non-split}*{\S1.2}.

The Theorem~\ref{th:main} is particularly relevant to one type of family of $L$-functions. Recall the fixed $\GL(2)$ automorphic cusp form $\pi$ and assume its central character to be trivial. The $L$-function $L(s,\pi\times \chi)$ may be defined via the Rankin-Selberg method. We choose the unitary normalization so that the functional equation links $L(s,\pi\times \chi)$ with $L(1-s,\pi \times \chi)$. We recall that $L(s,\pi\times \chi)$ is self-dual, of degree four and the sign of the functional equation is $\pm 1$. Under some general assumptions this sign is independent of $\chi$. Let $s=\Mdemi+it$ be on the critical line. The first named author investigated in~\cite{Temp:non-split} the asymptotic behavior of the first moment
\begin{equation}\label{intro:moment}
\frac{1}{h(D)}
\sum_{\chi \in \clchar}
L(\Mdemi+it,\pi \times \chi),
\qtext{as $D\to -\infty$.}
\end{equation}
When the sign of the functional equation is $-1$ one may consider the central derivative $L'(\Mdemi,\pi \times \chi)$ as well. 

Establishing a good error term in the sums~\eqref{eq:th:main} seems to be the most efficient method to handle the moments~\eqref{intro:moment}. Indeed the relationship between Theorem~\ref{th:main} and the moments~\eqref{intro:moment} is as follows. Up to some multiplicative factors, $L(s,\pi\times \chi)$ is closely related to the Dirichlet series
\begin{equation}\label{intro:dseries}
\sum_{\Fma \subset \CmO_D} \chi(\Fma) a_\pi(\TmN \Fma) \TmN \Fma^{-s}.
\end{equation}
Here $\Fma$ runs through the integral ideals of $K$. Let $Y$ be the analytic conductor of $\pi\times \chi$. When considering $L'(\Mdemi,\pi\times \chi)$the above series may be truncated to a weighted sum over the ideals $\Fma$ of norms $\TmN\Fma$ up to about $Y^{1/2+\epsilon}$ (approximate functional equation). This truncation is up to a negligible error term and $\epsilon>0$ may be chosen arbitrary small.

Now averaging~\eqref{intro:dseries} over $\chi\in \clchar$ has the effect of singling out the ideals $\Fma$ that are principal. A principal ideal is generated by an element $\frac{a+b\sqrt{D}}{2}$ in $\CmO_D$ with $a,b$ integers. Its norm equals $a^2+b^2\abs{D}$. We are therefore reduced to estimating the asymptotic of
\begin{equation}
\sum_{a,b\in \BmZ} a_\pi(a^2+b^2\abs{D}) W(\frac{a^2+b^2\abs{D}}{Y})
\end{equation}
for a certain truncation function $W$. The function $W$ does not exactly fulfill the requirement in Theorem~\ref{th:main}, but standard techniques enable to reduce to that case (partition of unity, dyadic subdivision). 

The case $b=0$ is special and yields the main term in the final asymptotic (see~\eqref{intro:derivative} below). When $b\neq 0$ we let $d=b^2\abs{D}>0$ and we recognize a sum similar to Theorem~\ref{th:main}. Note that $M_{\pi,d}=0$ because $\pi$ has trivial central character. After some more work which we omit here we obtain the following estimate
\begin{equation}\label{intro:derivative}
\frac{1}{h(D)}
\sum_{\chi \in \clchar}
L'(\Mdemi,\pi \times \chi)
=\alpha (\Mdemi \log \abs{D} + \frac{L'}{L}(1,\chi_D))
+ \beta + O(\abs{D}^{-\eta}),
\end{equation}
as $D\to -\infty$. Here $\alpha$, $\beta$ are explicit multiplicative factors. They behave essentially as constants, in the sense that we have that 
\begin{equation}
C^{-1}\le \alpha \le C,\quad  -C\le \beta \le C
\end{equation}
for some constant $C>1$ that depends only on $\pi$. 

This generalizes Theorem~2 in~\cite{Temp:non-split}. In~\cite{Temp:non-split} the asymptotic estimate~\eqref{intro:derivative} had been established (with the precise value of $\alpha$ and $\beta$ which we don't repeat here) under the following assumptions: the level of $\pi$ is required to be odd and square-free and the discriminant $D$ to be prime or almost prime.%
\footnote{Almost prime in the sense that the smallest prime divisor of $D$ be larger than $\abs{D}^\epsilon$ for some fixed $\epsilon>0$.} We are able to remove these assumptions here because of our main Theorem~\ref{th:main}.

An application of~\eqref{intro:derivative} is to quantitative non-vanishing. It may be proven that there are at least $\gg \abs{D}^\eta$ characters $\chi\in \clchar$ such that $L'(\Mdemi.\pi \times \chi)$ is nonzero. Here $\eta>0$ is an absolute constant. See also~\cite{Temp:height} for an alternative approach to this non-vanishing question.

\subsection{The exponents \texorpdfstring{$\theta$}{theta} and \texorpdfstring{$\delta$}{delta}}\label{sec:intro:sub} We recall briefly in this subsection what is currently known on bounds towards the Selberg eigenvalue conjecture and bounds for Fourier coefficients of half-integral modular forms.% which is a bound towards the Lindel\"of hypothesis for central values of $L$-functions in the level aspect.

Let $\phi$ be an Hecke-Maass cusp form for $\Gamma_0(N)$. It is an eigenfunction of the Laplacian, with eigenvalue $\lambda=1/4+ r^2$, where $1/2\pm ir$ are the 
Satake parameters for $\phi$. Selberg conjectured that $\lambda \geq 1/4$, and proved a nontrivial bound $\lambda \geq 3/16$. The best bound known thus
far is $|\Im r| \leq 7/64$ obtained by Kim-Sarnak~\cite{Kim-Sarnak} (or equivalently $\lambda \geq \frac{975}{4096}$). Let $0\le \theta <1/2$ be such a numerical value towards the Selberg conjecture.

We also need bounds for Fourier coefficients of half-integer weight cusp forms. Specifically, fix a cusp form of half-integer weight $f_j$. Note that it is orthogonal to the residual spectrum, spanned by theta functions of one variable for which the Fourier coefficients have different sizes. Assuming for simplicity that $f_j$ is holomorphic of weight $k+1/2$, it has a Fourier expansion of the form $f_j(z)=\sum_{n>0}\rho_j(n)e(nz)$ valid in the upper half plane. Often called the trivial estimate, we have
$|\rho_j(n)|\leq |n|^{k/2-1/4}$. The first breakthrough in the subject was an improvement by Iwaniec (later generalized by Duke) to $\rho_j(n)\leq |n|^{k/2-2/7}$.
We shall define $\delta > 0$ to be the smallest known constant such that $|\rho_j(n)|\ll_{\epsilon} |n|^{k/2-1/2+\delta+\epsilon}$. 

It is known through work of Waldspurger, and later Kohnen-Zagier that $\delta$ is related to a subconvexity problem. Specifically, if $\phi_j$ is the
integer weight modular form related to $f_j$ through the Shimura-Shintani correspondence, then there is a formula relating $f_j(n)$ and 
$L(\phi_j\times\chi_n,1/2)$ where $\chi_n$ is the Jacobi symbol of modulus $n$. Thus, the Lindel\"of hypothesis (and therefore also GRH) would imply that one
can take $\delta=0$.
 
 One therefore expects that $\delta$ is a much harder constant to study than $\theta$, not being directly related to a spectral problem
(unlike the case for Fourier coefficients of integral weight modular forms).

The best known bound thus far for $\delta$ has been obtained in an appendix of Mao to the subconvexity estimate by Blomer--Harcos--Michel~\cite{BHM07}, and so currently one can take $\delta = \frac14-\frac{1}{16}(1-2\theta)$. 

\subsection{Previous works}\label{sec:intro:works} Before explaining the details of the proof of Theorem~\ref{th:main} we review some of the previous works related to this problem. We try to proceed by chronological order.

Hooley~\cite{Hool63} has obtained a power saving in the remainder term of~\eqref{dirichlet} and the exponent was later improved by Bykovskii~\cite{Bykovskii} and Deshouillers-Iwaniec~\cite{DeI:82}. The starting point of the argument in~\cites{Hool63,Bykovskii,DeI:82} is the convolution identity $\tau=1*1$. This method is not applicable to general coefficients $a_f$ of modular forms and it is necessary to develop other methods. We now proceed to explain how these sums can be interpreted in the context of automorphic forms, following Sarnak~\cite{Sarn84}.

It is observed~\cite{Sarn84}*{p.295} that the series
\begin{equation} \label{seriestau}
\sum_{n=1}^{\infty}
\frac{\tau(d+n^2)}{(d+n^2)^{s-1/4}}
\end{equation}
is equal up to some Gamma factors to the integral
\begin{equation} 
\int_{\Gamma_0(4)\SB \FmH}
\theta(z) \overline{E(1/2,z)}
P_d(s,z) \frac{dxdy}{y^2}
\end{equation}
where $\theta(z)=y^{1/4}\sum_{n\in \BmZ}^{} e(n^2 z)$ is the standard theta function, $E(s,z)$ is the standard Eisenstein series on $\SL_2(\BmZ)$ and $P_d(s,z)$ is the $d$th Poincar\'e series of weight $1/2$ on $\Gamma_0(4)$ (se also~\S\ref{sec:half:poincare}). As a consequence the poles of the series~\eqref{seriestau} correspond to the spectrum of the Laplace operator on the space of forms of weight $1/2$ on $\Gamma_0(4)\SB \FmH$. The method is not pushed further in~\cite{Sarn84} but see \S~\ref{sec:hol} and \S\ref{sec:intro:poincare} below.

The reader will find in~\cite{Sarn84} some further discussions on the spectrum and the Selberg's bounds. We recall that the spectrum in integral and half-integral weight is related by the Shimura correspondence (see \S\ref{sec:half:shimura} below). We repeat here a clever observation from~\cite{Sarn84}*{pp.301--304} that there is a nice way to see the Shimura correspondence from the analytic properties of the series~\eqref{seriestau}. When $d=-h^2$ is minus a perfect square, the quadratic polynomial $n^2+d$ splits and we are reduced to shifted convolution sums of $\tau(n-h)\tau(n+h)$. The corresponding Dirichlet series is then related, up to some Gamma factors, to the integral
\begin{equation} 
\int_{\Gamma \SB \FmH} \abs{E(1/2,z)}^2 P_h(s,z) \frac{dxdy}{y^2}
\end{equation}
where $P_h(s,z)$ is a Poincar\'e series of integral weight. Thus the poles are related to the spectrum of the Laplacian on integral weight forms. This is essentially the same series with the same set of poles and this yields therefore a relation between the spectrum on integral and half-integral weight forms. This observation is compatible with the Maass form version of the Shimura correspondence.

We briefly recall the argument of Blomer~\cite{Blom08} to establish~\eqref{intro:blomer}. It is assumed that $f$ is holomorphic and therefore can be written as a linear combination of Poincar\'e series. This allows to replace the coefficients $a_f$ by sums of Kloosterman sums. Then the $n$-sum is evaluated by Poisson summation. This produces sums of Kloosterman sums of half-integral weight which are handled with the Kuznetsov's trace formula.

The approach in Templier~\cite{Temp:non-split} is based on the $\delta$-symbol method of Duke--Friedlander--Iwaniec to detect the quadratic progression $n^2+d$. This produces similar kind of exponential sums. The analysis in~\cite{Temp:non-split} and~\cite{Blom08} differs because the relative size of $d$ and $X$ are different. The argument in~\cite{Temp:non-split} proceeds by using period formulas to relate sums of exponential sums to special values of $L$-functions and then conclude from the Duke--Iwaniec bound for coefficients for half-integral weight forms (see~\S\ref{sec:half:iwaniec}).

Both approaches in~\cite{Temp:non-split} and~\cite{Blom08} are not well-suited to achieve an optimized value of the exponents because they rely on a large number of transformations. It is one of the purpose of the present paper to improve on the quality of the exponents.

\subsection{Approach with Poincar\'e series} \label{sec:intro:poincare}
The main idea in the proof of Theorem~\ref{th:main} is to interpret the sum on the left as the $d$'th Fourier coefficient of $\phi\overline\theta(z)$,
where $\phi$ is a Maass form corresponding to a new vector in $\pi$ and $\theta(z)$ is a suitable theta function 
(it should be clear from context whether $\theta$ is a function or the bound towards Ramanujan). There are then
two methods we pursue:

In section~\ref{sec:hol}, we proceed classically and use Poincar\'e series to isolate the $d$'th Fourier coefficient. Specifically, 
by taking the Petersson inner product of $\phi\overline\theta$ with an appropriate
Poincar\'e series $P_d(s,z)$, we can form the Dirichlet series:
\begin{equation}\label{ex:dirichlet}
D(s):=\sum_{n} \frac{a_\pi(n^2+d)}{\abs{n^2+d}^s}
\end{equation}

We then spectrally expand $P_d(s,z)$ into Maass-Hecke eigenforms $\phi_j(z)$, to get an identity of the form

\begin{equation}\label{ex:spectral}
D(s)=\sum_{j} \langle P_d(s,z),\phi_j(z)\rangle\cdot\langle \phi\overline\theta(z),\phi_j(z)\rangle.
\end{equation}
 
To apply \eqref{ex:spectral}, we need to bound the terms $\langle \phi\overline\theta(z),\phi_j(z)\rangle$. This is the crucial triple product estimate established by 
Sarnak in~\cite{Sarn94}. Next, by unfolding the integral we see that each term $\langle P_d(s,z),\phi_j(z)\rangle$ can be continued to $\MRe(s)>\frac12+\theta$, 
unless $\phi_j(z)$ is a 1-variable theta function in which case we get a pole at $s=\frac{3}{4}$. Also, we get a pointwise bound on $\langle P_d(s,z),\phi_j(z)\rangle$
of some simple factor times the d-th Fourier coefficients $\rho_j(d)$. We thus get the meromorphic continuation of $D(s)$ to $\MRe(s)>\frac12+\theta$, with
a simple pole at $s=\frac34$ coming from the exceptional eigenvalues of the Laplacian, and a bound on $D(s)$ in terms of $|d|^{\delta}$. 

Finally, in order to translate this information to the sum 
$$\sum_{n\ge 0} a_\pi(n^2+d) W(\frac{n^2+d}{Y})$$ we form the integral 
\begin{equation}\label{ex:integral}
\int_{\MRe(s)=1}\widetilde{W}(s)D(s)\frac{ds}{2i\pi}.
\end{equation}

Besides having complications for Maass forms, a serious downside of this approach is that to ensure convergence of \eqref{ex:integral} we need to insist that 
$\widetilde{W}(s)$ decays exponentially on vertical strips. We note that this approach has been developed independently by Hansen~\cite{Hansen} with a view towards number fields.

The second approach that we take is to spectrally expand immediately and use Sobolev norms as in Blomer-Harcos~\cite{BH08}, sidestepping the use of Poincar\'e series
and triple product estimates. The purpose of section~\ref{sec:sketch} is to demonstrate the main idea without getting caught up in the technical details.
 As such, in section~\ref{sec:sketch} we again restrict to holomorphic $\phi$ but now we spectrally expand $\phi\overline\theta$ directly 
without resorting to Poincar\'e series. This gives an expansion of the form
\begin{equation}\label{ex:spectralholo}
\phi\overline\theta = \sum_{\tau\in RES}c_\tau\psi_\tau + \sum_{\tau\not\in RES} c_\tau\psi_\tau
\end{equation}
where the sum ranges over distinct Maass forms of an appropriate weight.  Rather than use triple product estimates to bound the coefficients $c_{\tau}$
we use Sobolev norms as in Blomer-Harcos to establish convergence of the spectral expansion. Letting $K$ correspond to the weight of $\pi$, we then equate 
$d$'th Fourier coefficients to get the identity:

\begin{equation}\label{ex:coef}
Y^{\frac{K}{2}+\frac14}\displaystyle\sum_{n\in\Z}(n^2+d)^{\frac{K-1}{2}}a_\pi(n^2+d)e^{-\frac{(2n^2 + d)}{Y}} = \displaystyle\sum_\tau c_\tau\rho_\tau(d) W_{\frac{K}{2}-\frac14,it_\tau}(\frac{d}{Y})
\end{equation} 

From equation \eqref{ex:coef} we immediately read off Theorem ~\ref{th:main} for the particular class of test functions $W(x)=e^{-ax}x^{\frac{K-1}{2}}$. 

As we work with a space of Maass forms of fixed weight on the upper half plane rather than on the group, the required uniformity on the asymptotics of
Whittaker functions and bounds for Fourier coefficients of Hecke-Maass forms are much easier to obtain. Of course, the two advantages of working on the group
are that we can handle Maass forms as well as holomorphic forms, and the test function $W(x)$ shows up naturally as a Whittaker coefficient from the Kirillov
model. 

\subsection{Methods of proof}
As explained above the method with Poincar\'e series has the drawback that it doesn't work as well for Maass forms, due to a lack of harmonics. 

To handle Maass forms we work directly on the group $\SL_2(\BmR)$ rather than the upper half plane, which provides the missing
harmonics since now one is allowed to vary the weight of vectors in the representation $\pi$. Harcos called attention to this issue in his thesis, and this solution recently appeared in Blomer-Harcos~\cite{BH09} for the classical split shifted convolution problem, using ideas of Venkatesh~\cite{Venk05}. Also, they bypass the need for triple product estimates by using Sobolev norms, allowing for a softer treatment.

In our case, we have to go to the metaplectic cover $\widetilde \SL_2(\BmR)$ since that's where the theta functions naturally live. 

We choose a vector $\phi\in \pi$ whose Whittaker function matches the test function $W$ in Theorem~\ref{th:main}. Then the proof develops in the same way as the holomorphic case described above. We arrive at an expression similar to~\eqref{ex:coef}.

Since there is not a good enough theory of Kirillov models on the metaplectic group, we have to derive all our bounds directly. We cannot go back and forth from the abstract Kirillov model to the Whittaker coefficients. 

We use a uniform bound for the Whittaker function (Proposition~\ref{prop:W-bound}) which is a key feature. We display the typical case corresponding to principal series. For $p,r\in \BmR$,
\begin{equation}\label{intro:W-bound}
\frac{W_{p,ir}(y)}{\Gamma(\Mdemi+p+ir)}
\ll_\epsilon 
(\abs{p}+\abs{r}+1)^A y^{1/2-\epsilon},
\quad 0<y<1.
\end{equation}
The estimate was previously known for $p\in \BmZ$ and the existing proof does not extend to half-integer weights. We establish~\eqref{intro:W-bound} in \S\ref{sec:appendix} for all real weights $p\in \BmR$ using an integral representation and shifting contours. 

The estimate~\eqref{intro:W-bound} is suitable for application to our problem. The polynomial growth in the weight $p$ and the eigenvalue $r$ is important in relation to Sobolev norms and the spectral expansion. The behavior as $y\to 0$ is an essential feature. For the complementary series the exponent $1/2$ is replaced by $(1-\theta)/2$ which is directly related to the exponents of $d$ and $Y$ in Theorem~\ref{th:main}.

\subsection{Structure of the article}

The paper is organized as follows. In sections~\ref{sec:gl2} and \ref{sec:half} we give the necessary background on automorphic forms of integral and
half-integral weight, as well as the metaplectic group. We also record some estimates for the Whittaker functions which will be needed
in section~\ref{sec:pf}. In section~\ref{sec:hol} we work out the proof for holomorphic forms using Poincar\'e series. In section~\ref{sec:sketch} we sketch the proof of
Theorem 1 for holomorphic forms using representation theory and Sobolev norms, trying to keep  the presentation as classical as
possible so as to give the flavor of the argument. In section~\ref{sec:pf} we work out the details of the general case by going to the
metaplectic group and using the methods of~\cite{BH09}. Finally, the section~\ref{sec:appendix} gives the proof of Proposition~\ref{prop:W-bound}.

\subsection{Acknowledgments} We thank the organizers of the conference \Lquote{Equidistribution on homogeneous spaces} at Ohio State University in June 2008, 
from which this work originated. We thank Gergely Harcos and Peter Sarnak for helpful discussions and encouragement. The first named author would like
to thank the Institute of Advanced Study for providing a stimulating environment in which to work and acknowledges support from the NSF under agreement No. DMS-0635607.

%---------------
\section{Background on \texorpdfstring{$\GL_2$-}{}automorphic forms}\label{sec:gl2}

\subsection{Automorphic representations}\label{sec:gl2:rep} Let $G=\SL_2(\BmR)$ and $\Gamma=\Gamma_0(N)$ the Hecke congruence subgroup. Let $\chi$ be a Nebentypus character, namely a unitary congruence character on $\Gamma$, see~\eqref{def:autform}. We work on the space $\lcusp$ of cuspidal automorphic functions with Nebentypus $\chi$ acted upon by $G$ by translations from the right. Let $N\subset G$ be the unipotent subgroup $\Mdede{1}{\BmR}{0}{1}$. We recall that a cuspidal function $\phi$ is such that
\begin{equation}
\int_{\Gamma \cap N \SB N} \phi(ng) dn=0,
\qtext{for a.e. $g\in G$}.
\end{equation}
We have a Hilbert direct sum of irreducible cuspidal $G$-representations:
\begin{equation}
\lcusp=\bigoplus_{\pi} V_\pi.
\end{equation}
We arrange so that each space $V_\pi$ is preserved by the Hecke operators.

We let $K=SO(2)$ be the usual maximal compact subgroup. We may consider the restriction of the $G$-representation $V_\pi$ to $K$ and decompose further according to unitary characters of $K$:
\begin{equation}
V_\pi=\bigoplus_{k\in \BmZ}
V_{\pi,k}.
\end{equation}
More precisely let $k_\theta=\Mdede{\cos \theta}{\sin \theta}{-\sin \theta}{\cos \theta}$ be a generic rotation. A vector is of pure weight $k$, when $k_\theta$ acts on it through the character $\theta \mapsto e(k\theta)$.

We recall the classification of irreducible unitary representations of $\SL_2(\BmR)$, see e.g.~\cite{book:Lang:sl2}. The Casimir operator
\begin{equation}
\Delta = y^2(\frac{\partial^2}{\partial x^2}+
\frac{\partial^2}{\partial y^2})
- y\frac{\partial}{\partial x}
\frac{\partial}{\partial \theta}
\end{equation}
acts by a positive scalar
\begin{equation}
\lambda_\pi = \frac{1}{4}+r_\pi^2,
\end{equation}
on the whole irreducible space $V_\pi$. The discrete series correspond to $ir_\pi \in \Mdemi+\BmZ$; the principal series correspond to $r_\pi \in \BmR$ and the complementary series correspond to $r_\pi \in [-\frac{i}{2},\frac{i}{2}]$.

\subsection{Maass forms of weight zero}\label{sec:gl2:maass}
Let $\phi\in \lcusp$ be a Hecke-Maass cusp form of weight zero on $\Gamma=\Gamma_0(N)$. The line generated by $\phi$ is equal to $V_{\pi,0}$ for a unique automorphic cuspidal representation $\pi$. The weight zero Laplacian reads $
 \Delta=y^2(\Mpartialsquare{x}+\Mpartialsquare{y})$, and the following holds:
\begin{equation}
 \Delta \phi + (\frac{1}{4}+r^2) \phi =0.
\end{equation} 
The spectral parameter $r$ belongs to $\BmR \cup [-\frac{i}{2},\frac{i}{2}]$.
The Fourier expansion might be written in the following way:
\begin{equation}
 \phi(z)=2\sum_{n\not=0} \frac{a(n)}{\sqrt{n}} W_{0,ir}(4\pi|n|y)e(nx),
 \quad 
 z=x+iy\in\FmH.
\end{equation}
Here $\FmH=G/K$ is the upper-half plane. 

Since $\phi$ is fixed throughout the paper, we shall be omitting the subscript and denote simply by $a(n)$ the Fourier coefficients. It will not cause any ambiguity.
We normalize $\phi$ so that $a(1)=1$. By convention, $a(0)=0$. 

The Hecke bound reads $a(n)\ll |n|^{1/2}$. Under the Ramanujan conjecture $a(n)\ll
|n|^\epsilon$ would hold. We have $a(n)\ll |n|^\theta$, for all $\theta>7/64$ is achieved in~\cite{Kim-Sarnak}. 

More detailed properties of Whittaker functions will be recalled below. We recall the following (\cite{GR}*{(9.235)})
\footnote{in the notation $W_s$ of \cite{Iwaniec}*{(1.26)}, $W_{0,ir}(4\pi
y)e(x)
 =W_s(z)$, here we follow the notations in~\cite{KS}}
\begin{equation}
 W_{0,ir}(4\pi y)
 =2 y^{1/2} K_{ir}(2\pi y)
\end{equation}
where $K_{ir}$ is the $K$-Bessel function. Thus an equivalent expression for the Fourier expansion of $\phi$ is:
\begin{equation}
\phi(z)=y^{1/2}\sum_{n\neq 0} a(n)K_{ir}(2\pi|n|y)e(nx).
\end{equation}

When $\phi$ is a newvector, the $L$-function associated to $\pi$ is:
\begin{equation*}
 L(s,\phi):=\sum^\infty_{n=1} \frac{a(n)}{n^{s}}
\end{equation*}
We recall that the functional equation relates the value at
$s$ to the value at $1-s$. For more details see \cite{Michel-park}*{p.41}
and \cites{KS,LRS,book:Gelbart}.

\remark Another interesting choice for $\phi$ would be the Eisenstein series
where $a(n)$ gets replaced by $\tau(n)$, the divisor function. It
is noteworthy that in this case would cover the result of
Hooley~\cite{Hooley63} in~\eqref{intro:aim} as observed in~\cite{Sarn84}. The constant term yields another main term.

\subsection{Holomorphic modular forms}\label{sec:gl2:holomorphic} For an integer $K\ge 1$, we consider the complex vector space of weight $K$ holomorphic cusp forms on $\Gamma \SB G$. These are bounded automorphic functions $F$ on the upper-half plane which satisfy the automorphy relation:
\begin{equation}
F(\gamma z)=J(\gamma,z)^{2K} F(z),
\quad
\forall z\in \FmH, \gamma\in \Gamma.
\end{equation}
The cocycle $J(\gamma,z)$ is classical and will appear below. Here we recall that $J(\gamma,z)^2= cz+d$ for $\gamma=\Mdede{a}{b}{c}{d}\in \Gamma$.

Let $F$ be an holomorphic Hecke cusp form of weight $K$.
To have consistent notations, it is good to work with $f(z)=y^{K/2}F(z)$. The Fourier expansion reads:
\begin{equation}\label{Fourier-holomorphic}
f(z)=\sum_{n\ge 1} 
\frac{a(n)}{\sqrt{n}} 
W_{\frac{K}{2},\frac{K-1}{2}}
(4\pi n y) e(nx).
\end{equation}
We recall that $W_{\frac{K}{2},\frac{K-1}{2}}(y)=y^{K/2} e^{-y/2}$, so that~\eqref{Fourier-holomorphic} is equivalent to the usual $q$-development of $F$ inside the cusp at infinity. The function $f$ may be lifted to $\lcusp$ in the usual way and there corresponds a unique automorphic cuspidal representation $\pi$. The function $f$ then belongs to the line $V_{\pi,K}$ of weight $K$ vectors in $V_\pi$. The spectral parameter is $r_\pi=\frac{K-1}{2}$.

\subsection{Kirillov models}\label{sec:gl2:kirillov} More generally than the classical examples recalled in~\S\ref{sec:gl2:maass} and \S\ref{sec:gl2:holomorphic} it will be important in the sequel to choose an arbitrary form $\phi\in L^2(\Gamma\SB G)$ with mixed weights. Precisely we shall choose $\phi\in V_\pi$ in such a way that the Whittaker function that occurs in its Fourier expansion precisely matches the smoothing function $W$ in Theorem~\ref{th:main}.

We shall use the notation
\begin{equation*}
n(u)=\Mdede{1}{u}{0}{1},\quad
a(y)=\Mdede{y^{1/2}}{0}{0}{y^{-1/2}},
\end{equation*}
as elements of $\SL_2(\BmR)$. For a smooth function $\phi \in V^\infty_{\pi}$ we have the Whittaker integral
\begin{equation*}
W_\phi(y)=\int^{1}_{0} \phi(n(u)a(y))e(-u)du,
\quad y>0.
\end{equation*}
As in~\cite{BH08}, we define the Sobolev norms as
$\pnorm{\phi}_d = \sum_\mathfrak{D} \pnorm{\mathfrak{D}\phi}$ where $\mathfrak{D}$ ranges over all monomials in $H,L,R$ of degree at most $d$.

\begin{proposition}\label{prop:wtransform}
Let $W\in \Ccinf(0,\infty)$ be a smooth function of compact support and $\pi$ an automorphic representation. There exists a function $\phi\in V^\infty_\pi$ such that
\begin{equation}
W_\phi(y)=W(y),\quad y>0.
\end{equation}
\end{proposition}
The proposition follows from the surjectivity of the Kirillov model.
When $\Gamma=\SL_2(\BmZ)$ the details are carried out in~\cite{BH08}*{\S 2.3} and we do not repeat it here. In particular note that $\pnorm{\phi}_d<\infty$ for all integer $d\ge 1$. Also if $W^{(i)}\ll 1$ for all $i\ge 0$ then $\pnorm{\phi}_d\ll 1$ for all $d$. Since $\pi$ will be fixed, this inequality will suffice in the sequel. 

\subsection{New vector}
We have the Fourier expansion
\begin{equation}
\phi(n(x)a(y)) = \sum_{n\neq0} \frac{a_\phi(n)}{\sqrt{n}}W_\phi(ny)e(nx)
,\quad 
x\in \BmR,\ y\in \BmR_{>0}.
\end{equation}
The Theorem~\ref{th:main} is valid more generally for the Fourier coefficients $a_{\phi}(n)$ as the proof in section~\ref{sec:pf} shows. We need to explain that the form $\phi$ may be chosen in such a way that $a_\phi(n)$ coincide with the Dirichlet coefficients $a_\pi(n)$ of the associated $L$-function.

This is a classical fact from the theory of newvectors. When $\phi$ is a newvector the Fourier transform of the global Whittaker function coincides with the normalized $L$-function $L(s,\pi)$. See e.g. \cite{book:Gelbart}*{(6.37)} and~\cite{book:Gelbart}*{Thm~1.12} in the holomorphic case.

\section{Background on automorphic forms of half-integer weight}\label{sec:half}

\subsection{Half-integer forms}\label{sec:half:intro} A standard reference for
holomorphic forms is Shimura~\cite{Shimura73}, see also the introduction of~\cite{Duke88} and the references herein. We conserve notation consistent with~\cite{KS}.

The standard theta series reads:
\begin{equation}
 \theta(z):=y^{1/4}\sum_{n\in \BmZ} e(n^2 z).
\end{equation}
We may define the cocycle multiplier as follows:
\footnote{Explicitly one has:
\begin{equation}
 J(\gamma,z)=\chi_4(\gamma)e(i\arg(cz+d)/2),
 \quad
  \gamma=\Mdede{a}{b}{c}{d} \in \Gamma_0(4), z\in \FmH,
\end{equation}
where $\arg(cz+d)\in (-\pi,\pi]$ and $\chi_4$ is as in~\cite{Shimura73} or~\cite{Duke88}*{(2.1)}.}
\begin{equation}
 J(\gamma,z):=\frac{\theta(\gamma z)}{\theta(z)},\quad \gamma \in \Gamma_0(4).
\end{equation}

Let $k\ge \Mdemi$ be an half-integer with $2k$ odd. The Laplacian of weight $k$ is:
\begin{equation*}
\Delta_{k}=y^2(\Mpartialsquare{x}+\Mpartialsquare{y})
 -k iy\Mpartial{x}.
\end{equation*}

For $\chi:(\mathbb{Z}/N\mathbb{Z})^\times \to \BmC^\times$ a Dirichlet character modulo $N$, we define an automorphic cusp form $f$ of weight $k$ and level $N$ to be a function which satisfies:
\begin{equation}\label{def:autform}
  f(\gamma z)=\chi(a)J(\gamma,z)^{2k} f(z),\quad \forall \gamma= \begin{pmatrix} a& b\\c & d \end{pmatrix}\in \Gamma_0(4N),
\end{equation}
and
\begin{equation}
 \Delta_k f +\lambda f=0,\quad \lambda =\frac{1}{4}+r^2.
\end{equation}

To keep track on the dependency on $r=r_j$ and $k$ we let $f=f_{j,k}$, $\lambda=\lambda_j$ and assume as in Duke~\cite{Duke88} that it is $L^2$-normalized, $\pnorm{f_{j,k}}=1$  (note that the normalization in~\cite{KS} is such that $\rho_{j,k}(1)=1$). 

Its Fourier expansion reads as follows:
\begin{equation}\label{Fourier-half}
 f_{j,k}(z)=\sum_{n\not =0} \rho_{j,k}(n) W_{\Msgn(n)k/2,ir_j}(4\pi |n| y) e(nx),
 \quad z\in \FmH.
\end{equation}

Similarly, we define the space of holomorphic functions $S_k(N,\chi)$ of weight $k\ge 1/2$, level $4N$ and Nebentypus $\chi$ to consist of all holomorphic
functions $F(z)$ satisfying
\begin{equation}
F(\gamma(z))=\chi(a)(cz+d)^{k/2}\left(\frac{c}{d}\right)^k\left(\frac{-1}{d}\right)^{k/2}F(z),
\end{equation}
for all $\gamma= \begin{pmatrix}a&b\\c&d\end{pmatrix}\in\Gamma_0(4N)$.

If $F(z)\in S_k(N,\chi),$ then $y^{k/2}F(z)$ is an automorphic form of weight $k$, level $N$, 
Nebentypus $\chi$ and eigenvalue $\frac{k}{2}(1-\frac{k}{2})$, in the above sense. The Fourier expansion~\eqref{Fourier-half} is supported on the integers $n\ge 1$ and we recall that the Whittaker function again is: $W_{\frac{k}{2},\frac{k-1}{2}}(y)=y^{k/2}e^{-y/2}$.  
If $F(z)$ is moreover a newform, then this corresponds to the discrete series with $ir_j=\frac{k-1}{2}$. In this case $F(z)$ is the smallest weight vector, so
that $f_{j,p}=0$ when $\abs{p}<k$.

We discuss briefly the anti-holomorphic forms since they appear in the spectrum as well. Anti-holomorphic forms can be viewed as complex conjugates of holomorphic forms and therefore the analysis is no different. The Fourier expansion~\eqref{Fourier-half} is then supported on the non-positive integers $n$, the Fourier coefficients $\rho_{j,k}(n)$ are complex conjugate of the Fourier coefficients of holomorphic forms. The weight $k$ is negative and the spectral parameter $r_j$ is identical, in particular the Whittaker function $W_{\Msgn(n)k/2,ir_j}$ remains the same.

\subsection{Whittaker functions}\label{sec:gl2:whittaker} Throughout the text, the Whittaker function $W_{p,\nu}$ is as defined in section 9.2 of~\cite{GR}. We state in this paragraph the estimates we shall need during the proof of the main theorems. The proofs are given in \S\ref{sec:appendix}.

The crucial estimates concern the behavior as $y$ goes to zero. We shall distinguish three cases for the sake of clarity. The cases (i) (resp. (ii), (iii)) correspond to principal series (resp. complementary series, discrete series). These estimates were previously known for $p\in\BmZ$.

\begin{proposition}\label{prop:W-bound} There is an absolute constant $A>0$ with the following properties. 
(i) For $p,r\in \BmR$ the following holds:
\begin{equation}\label{eq:W-bound:i}
\frac{W_{p,ir}(y)}{\Gamma(\Mdemi+p+ir)}
\ll_\epsilon 
(\abs{p}+\abs{r}+1)^A y^{1/2-\epsilon},
\quad 0<y<1.
\end{equation}

(ii) For $p\in \BmR$, $0<\nu<1/2$ the following holds:
\begin{equation}\label{eq:W-bound:ii}
\frac{W_{p,\nu}(y)}{\Gamma(\Mdemi+p)}\ll_\epsilon (\abs{p}+1)^A y^{1/2-\nu-\epsilon},
\quad 0<y<1.
\end{equation}

(iii) For $p,\nu \in \BmR$ with $p-\nu- \Mdemi \in \BmN$ and $\nu > -\Mdemi+\epsilon$ or the following holds:
\begin{equation}\label{eq:W-bound:iii}
\frac{W_{p,\nu}(y)}{\abs{\Gamma(\Mdemi+p-\nu)\Gamma(\Mdemi+p+\nu)}^{1/2}}\ll_\epsilon (\abs{p}+\abs{\nu}+1)^A y^{1/2-\epsilon},
\quad 0<y<1.
\end{equation}

All three upper bounds holds true uniformly in $p,\nu$ in the given region and the implied constant depends on $\epsilon>0$ only.
\end{proposition}

\remark We point out that the condition $p-\nu\in \Mdemi+\BmN$ in (iii) is necessary because otherwise the behavior as $y\to 0$ is different, as can be seen from the proof in~\S\ref{sec:appendix:pf} because extra residues would appear.

\remark In the integral weight case ($p\in \BmZ$), our results correspond to the bounds of Bruggeman-Motohashi~\cite{BM05}, Harcos-Michel~\cite{HM06} and Blomer-Harcos~\cite{BH08}, except for a worse exponent in $p$. The approaches in the above mentioned papers do not seem to generalize to all $p\in\BmR$. Our method in~\S\ref{sec:appendix} has the advantage of working for half-integer weights, which is crucial for the present paper. It yields a distinct proof in the integral weight case as well.

\subsection{Half-integer Eisenstein series}\label{sec:half:eisenstein}

We include here some background and properties of Eisenstein series of half-integer weight. More details can be found in~\cite{Duke88}*{section 2}.
 For $k$ a half-integer and $\MRe(s)>1$, we define 

\[
E_{N,k}(s,z)=\sum_{\Gamma_{\infty}\backslash\Gamma_0(4N)}\chi_N(\gamma) e^{ik\cdot\textrm{arg}(J(\gamma,z))}\MIm(\gamma z)^s
\]
where $\gamma = \begin{pmatrix} a & b\\ c & d \end{pmatrix},$ 
$\chi_N(\gamma) = \left(\frac{-1}{d}\right)^k\left(\frac{c}{d}\right)^{2k} \left(\frac{N}{d}\right)$ and $j(\gamma,z)=cz+d$. 
Set also $E_{1/2}(s,z)=E_{1,1/2}(s,z)$. 

The $E_{N,k}(s,z)$ furnish the continuous spectrum of Maass forms of weight $k$ and level $N$. It is possible to meromorphically continue $E_{N,k}(s,z)$ to the
whole complex plane with no poles in $\MRe(s)>1$ except for possibly at $s=3/4$. In particular, following Duke we have the identity
$$\textrm{Res}_{s=3/4}E_{N,1/2}(s,z) = 2e(-1/4)\pi^{3/4}N^{-1/2}y^{1/4}\sum_{n\in\mathbb{Z}}e(Nn^2z).$$

\subsection{Computing the $L^2$ norms of theta-functions}\label{sec:normtheta}

We include here a computation of the $L^2$ norm of the theta function $\theta_N(z)=y^{1/4}\sum_{n\in\Z}e(Nn^2z)$ on $\Gamma_0(4N)$.
The computation should be known, but as we were unable to locate a convenient reference in the literature we provide a brief proof. First
recall that $E_{N,0}(s,z)$ has a residue of $\frac{3}{\pi\cdot[\SL_2(\Z):\Gamma_0(4N)]}$ at $s=1$. We would now like to form 
$$\langle \theta_N(z)\overline{\theta_N}(z),E_{N,0}(s,z)\rangle$$ and take the residue at $s=1$. The issue is that the integral diverges at the cusp, so we must regularize.
Define \[E^T_{N,0}(s,z):=\sum_{\Gamma_{\infty}\backslash\Gamma_0(4N)}\chi_N(\gamma) e^{ik\cdot\textrm{arg}(j(\gamma,z))}\MIm(\gamma z)^s\delta_{\MIm(\gamma z)<T}.
\]
The function $E^T_{N,0}(s,z)$ is only different from $E_{N,0}(s,z)$ at the cusp, and as $T\rightarrow\infty$ we have $E^T_{N,0}(s,z)\rightarrow E_{N,0}(s,z)$.

Unfolding, we have
\begin{align*}
R_T(s,z)&=\langle\theta_N(z)\overline{\theta_N}(z),E^T_{N,0}(s,z)\rangle\\
&=\int_0^T y^{-3/2+s}dy + \sum_{n\neq 0}\int_0^T e^{-4\pi Nn^2y}y^{-1/2+s}dxd^{\times}y\\
&=\frac{T^{s-1/2}}{s-1/2} + \sum_{n\neq 0}\int_0^T e^{-4\pi Nn^2y}y^{-1/2+s}dxd^{\times}y\\
\end{align*}

Now, we have that $\langle\theta_N,\theta_N\rangle=\frac{\pi\cdot[\SL_2(\Z):\Gamma_0(4N)]}{3}\lim_{T\rightarrow\infty}Res_{s=1}R_T(s,z)$. The key observation is that when taking residues at $s=1$ the first term drops off and then we can interchange the residue with the limit so that
\begin{align*}
\langle\theta_N,\theta_N\rangle &=\frac{\pi\cdot[\SL_2(\Z):\Gamma_0(4N)]}{3}\cdot Res_{s=1}\left(\sum_{n\neq 0}\int_0^{\infty} e^{-4\pi Nn^2y}y^{-1/2+s}dxd^{\times}y\right)\\
&=\frac{\pi\cdot[\SL_2(\Z):\Gamma_0(4N)]}{3}\cdot Res_{s=1}\left(2(4\pi N)^{1/2-s}\zeta(2s-1)\Gamma(s-1/2)\right)\\
&= \frac{4\pi\cdot[\SL_2(\Z):\Gamma_0(4N)]}{3}\cdot (4\pi N)^{-1/2}\Gamma(-1/2)\\
\end{align*}

If moreover $N$ is square-free, then we know by the next subsection that $Res_{s=3/4}E_{N,k}(s,z)$ is a multiple of $\theta_N(z)$. To determine the multiple,
we can define $E^T_{N,k}(s,z)$ analogously to $E^T_{N,0}(s,z)$ and compute $Res_{s=3/4}\langle\theta_N(z),E^T_{N,k}(s,z)\rangle$ as above to be $1$. We thus have that
\begin{equation}\label{eisensteinresidue}
Res_{s=3/4}E_{N,k}(s,z) = \frac1{\langle\theta_N,\theta_N\rangle}\cdot \theta_N(z).
\end{equation}

\subsection{The residual spectrum}\label{residualspectrum}\label{sec:half:residual}

In this subsection we discuss the residual spectrum of half-integer weight. These are the Maass forms of weight $\frac{k}{2}$ which have eigenvalue $\lambda=\frac{3}{16}$ and occur as residues of Eisenstein series. We mention that the Eisenstein series defined in the previous subsection correspond
to the cusp at $\infty$. To get the entire residual spectrum, one has to consider Eisenstein series corresponding to every cusp into account here. These forms are important to us as they will contribute the main term. 

For $k>5$, the spectrum is gotten from the residual spectrum for $\left(k\mod(4)\right)$ by applying the raising operator, see~\S\ref{sec:half:op}. It is thus only necessary to discuss the case of $k=1$ and $k=3$.

For $k=3$ there is no residual spectrum. Briefly, the constant term of the Eisenstein series is (See \cite{Duke88}*{\S2}, (2.8)\footnote{
Note that Duke writes $k$ for what in our notation is $k/2$})

$$\displaystyle\frac{\pi^s4^{1-s}e(-k/8)\Gamma(2s-1)}{\Gamma(s+k/4)\Gamma(s-k/4)}\phi_l(s)$$

where $\phi_l(s)$ is a singular series with a potential simple pole at $s=3/4$. However, for $k=3$ the pole gets canceled by the pole of $\Gamma(s-3/4)$. For more details, see~\cite{Duke88}*{\S2}.

For $k=1$ the spectrum was described completely in a beautiful paper of Serre and Stark~\cite{Se-St}. The results are as follows.

Define $M_{\frac{1}{2}}(N,\chi)$ to be the space of modular forms of level $N$, weight $\frac12$, Nebentypus $\chi$ and eigenvalues $\frac{3}{16}.$ Then an orthogonal basis for $M_{\frac{1}{2}}(N,\chi)$ is given by the set of theta functions \[\theta_{\psi,t} = y^{\frac14}\displaystyle\sum_{n\in\mathbb{Z}} \psi(n)e(tn^2z)\] 
Here $t\in\mathbb{N},$ and $\psi$ is a Dirichlet character of conductor $L$ which satisfy $4L^2t$ divides $N$.

Note that the condition on $t$ and $\psi$ ensures the space is finite dimensional.

\subsection{The metaplectic group and theta series}\label{sec:half:group}

Let $\widetilde{G}$ denote the metaplectic group, a nontrivial central extension of $G=\SL_2(\BmR)$. We have the exact sequence
\begin{equation}
 \BmZ/2\BmZ \rightarrow \widetilde{G} \rightarrow G
\end{equation}
There are several ways to define the group $\widetilde{G}$. We recall here that the $2$-cocycle has an explicit form on the standard Borel subgroup of $\GL_2(\BmR)$~\cite{book:Gelbart}*{\S 2.1 and \S4.1}:
\begin{equation}
 \beta(\Mdede{a}{x}{0}{b},
 \Mdede{a'}{x'}{0}{b'})
 =
 (a,b')_\infty.
\end{equation}
The Hilbert symbol is as follows~\cite{book:IK04}*{(3.37)}:
\begin{equation}
 (a,b)_\infty=
 \begin{cases}
-1 & \text{if $a<0$ and $b<0$,}\\
 1 & \text{otherwise.} 
 \end{cases}
\end{equation}

In particular the extension $\widetilde G\to G$ splits over the subgroup
\begin{equation}
 NA=\{n(x)a(y), x\in\BmR, y\in \BmR^\times\},
\end{equation}
where
\begin{equation}
  n(x):=\Mdede{1}{x}{0}{1} \qtext{and}\quad a(y)=\Mdede{y^{1/2}}{0}{0}{y^{-1/2}}.
\end{equation}

Let $Z=\{\pm 1\}$ be the center of $\SL_2(\BmR)$. The extension $\widetilde G \to G$ splits over $Z^0$. The center of $\widetilde{G}$ is $Z(\widetilde{G})=\BmZ/2\BmZ$.

One may check that the extension $\widetilde G\to G$ splits over $\Gamma_0(4)$ \cite{Shimura73}. We denote by $\widetilde{\Gamma}\to \Gamma\simeq \Gamma_0(4N)$ the image of $\Gamma_0(4N)$ under the splitting.

According to~\cite{book:Gelbart}*{p.50} the automorphic form $\theta$ is a form on $\tilde G$. It transforms non-trivially under the center $\BmZ/2\BmZ$ and therefore it is  genuine. By definition a character on $Z(\widetilde G)$ (resp. automorphic form on $\widetilde G$) is genuine when it is non-trivial (resp. when it does not factor through $G$). Note that if an automorphic form is not genuine, then it transforms trivially under the center and is induced from an automorphic form on $G$. More precisely, consider the double cover $\widetilde{SO}(2)$ of $SO(2)$ in $\widetilde G$, which is just $S^1$ as an 
abstract group. If we have an automorphic form $\phi$ which transforms under a character $\chi_\phi$ of $\widetilde{SO}(2)$, we say that it is of weight $k$ if
$\chi_{\phi}=\chi_{\theta}^{2k}$, for $k$ a half-integer. It is now easy to see that a form is of half-integer weight iff it is genuine.

Let $\chi$ be a congruence character on $\widetilde \Gamma$ which is non-trivial on the center $\BmZ/2\BmZ$. We shall work on the space $\lgen$ of functions on $\widetilde G$ invariant under $(\widetilde \Gamma,\chi)$. Note that since $\chi$ is nontrivial on $\BmZ/2\BmZ$, all the representations of $\widetilde{G}$ that occur in $\lgen$ are genuine.

Finally, we mention how Fourier expansion works in the Metaplectic group. We parametrise $SO(2)$ by 
$k(\theta):=\begin{pmatrix} \cos(\theta)&\sin(\theta) \\ -\sin(\theta)& \cos(\theta)\end{pmatrix},$ where $0\leq\theta<2\pi$. We can parametrize 
$\widetilde{SO}(2)$ by $\widetilde{k(\theta})$, where $0\leq\theta< 4\pi.$ Using $\tilde{G}=N\cdot A\cdot\widetilde{SO}(2)$ we thus have the following Fourier expansion
for an automorphic form $f_{j,k}$ of weight $k$ and Laplacian eigenvalue $\lambda_j = \frac14 + r_j^2$

$$f_{j,k}(n(x)a(y)\widetilde{k(\theta)}) = e(k\theta)\sum_{n\in\mathbb{Z}} \rho_{j,k}(n)W_{\Msgn(n)k/2,ir_j}(4\pi|n|y)e(nx).$$

\subsection{Spectral decomposition}\label{sec:half:basis} Summarizing we have described the spectral decomposition of $\lgen$ which consists of the following.
\begin{enumerate}[(i)]
\item An orthonormal basis of cusp forms $f_{j,k}$, where $k$ is the weight and $r_j$ is the spectral parameter;
\item an orthogonal basis of residual forms $\theta_{\psi,t}$ which are generated by theta series as described in~\S\ref{sec:half:residual};
\item a continuous spectrum provided by the analytic continuation of Eisenstein series.
\end{enumerate}

\subsection{Maass operators}\label{sec:half:op} 
We take the usual basis for the lie algebra $\mathfrak{g}$ of $\SL_2(\BmR)$ as follows:

$$H=\begin{pmatrix} 1&0\\ 0&-1\end{pmatrix}, R=\begin{pmatrix} 0&1\\ 0&0\end{pmatrix},L=\begin{pmatrix} 0&0\\ 1&0\end{pmatrix}.$$

The center of the universal enveloping algebra $U(\mathfrak{g})$ is generated by the Casimir operator $\Delta:=H^2+2RL+2LR$. The operators $R, L$ have the
property that if $\psi$ is an automorphic form of weight $k$, then $R\psi, L\psi$ are of weights $k+2,k-2$ respectively. If we restrict these operators to
automorphic forms of pure weights viewed on the upper-half plane, we get the following classical operators.

The Maass lowering operator is defined by:
\begin{equation}
 \Lambda_k=iy
 \frac{\partial}{\partial x}-
 y\frac{\partial}{\partial y}+\frac{k}{2}.
\end{equation}
Suppose $f,g$ are of compact support. We have the following equality which follows by integration by parts:
\begin{equation}
 (f,\Delta_k g)=(\Lambda_k f ,\Lambda_k g)+\frac{k}{2}(1-\frac{k}{2})(f,g).
\end{equation}
The Maass raising operator is
\begin{equation}
 R_k:=iy\frac{\partial}{\partial x}
 +y\frac{\partial}{\partial y}
 +\frac{k}{2}.
\end{equation}
See~\cite{DFI8} and~\cite{Sarn84} for further discussions on the spectrum of $\Delta_k$.

If the $L^2$-normalized Maass cusp form $f_{j,k}(z)$ has spectral parameter $r_j$, weight $k$ and Fourier coefficients $\rho_{j,k}(n)$ then the normalized form $\dfrac{\Lambda_k f_{j,k}}{\pnorm{\Lambda_k f_{j,k}}}$ is of weight $k-2$ and its Fourier coefficients equal $\pnorm{\Lambda_k f_{j,k}}^{-1}\rho_{j,k}(n)$ when $n<0$ and $\pnorm{\Lambda_k f_{j,k}}\rho_{j,k}(n)$ when $n>0$. We assumed implicitly that $\pnorm{\Lambda_k f_{j,k}}$ is non-zero that is $f_{j,k}$ is not the lowest weight vector of a discrete series. In general one has 
\begin{equation}\label{raisingfactor}
 \pnorm{\Lambda_k f_{j,k}}^2=
 (\frac{k}{2}-\frac{1}{2}-ir_j)
 (\frac{k}{2}-\frac{1}{2}+ir_j)
\end{equation}

This follows, when $n>0$, by inspecting the Fourier expansion~\eqref{Fourier-half} and the following recurrence relation (\cite{GR}*{(9.234.3)}):
\begin{equation}
 (\nu^2-(p-\Mdemi)^2)W_{p-1,\nu}(y)=
 (p-\Mdemi y)W_{p,\nu}(y)
 -yW'_{p,\nu}(y).
\end{equation}
When $n<0$, one needs to use the following equality:
\begin{equation}
 W_{p+1,\nu}(y)=
 (p+\frac{y}{2})W_{p,\nu}(y)
 -yW'_{\lambda,\nu}(y).
\end{equation}
That latter equation may be proved starting from the Hankel's representation of the Whittaker function.

\subsection{Shimura correspondence and Selberg's bound}\label{sec:half:shimura}
We shall need a bound on the spectral parameter of half-integral weight automorphic forms. That is, let $\tilde{\pi}$ be an cuspidal representation on the 
metaplectic group $\widetilde{G}$, which does not correspond to a 1-dimensional theta function. We shall need lower bounds on the Laplacian eigenvalue 
$\lambda_{\tilde{\pi}}$. In order to do this one can use the Theta correspondence from $\widetilde{SL}_2$ to $PGL_2$. For an introduction, 
see ~\cite{Pras93}*{Theorem 8.7},\cite{KS}*{Proposition 2.3}, \cite{Sarn84}, and in the general case Waldspurger\cite{Wald80}. 

   As is explained in the Appendix to \cite{PS84}, 
given $\tilde{\pi}$ one can associate through the theta correspondence 
a \emph{non-zero} cuspidal representation $\pi$ on $PGL_2$,  and the representation $\pi$ at the infinite place depends only on representation $\tilde{\pi}$ at 
the infinite place (in Piatetski-Shapiro, the set of cuspidal metaplectic representation not coming from one dimensional theta-functions is referred to 
as $A_{00}$). We shall use only the following fact: If the spectral parameter of $\tilde{\pi}$ is $r$, then the spectral parameter of $\pi$ is $2r$, as is 
explained thoroughly in Gelbart~\cite{book:Gelbart}, Section 4.3. Recall
that the Laplacian eigenvalues are then $\lambda_{\pi}=\frac14 - 4r^2$ and $\lambda_{\tilde{\pi}} = \frac14 - r^2$. This enables us to transfer bounds towards Selberg's eigenvalue conjecture from integral weight to half integral weight. In particular, Selberg's $\lambda_{\pi}\geq 3/16$ bound corresponds to the Goldfeld-Sarnak~\cite{GS83} $\lambda_{\tilde \pi} \ge 15/64$ bound on half-integral weight.

%a form $F$ which is an eigenfunction of $\Delta_{1/2}$ with $LF=F$. The operator $L=\tau_2\sigma$ is
%defined on page 195 of \cite{KS}. Important for us is the
%relation between the eigenvalues, namely
%\begin{equation*}
% \Delta \phi + (\frac{1}{4}+(2t)^2) \phi =0
% ~;~
% \Delta_{1/2} F +(\frac{1}{4}+t^2) F =0
%\end{equation*}

\subsection{Iwaniec's bound}\label{sec:half:iwaniec} 
For cusp forms of half-integral weight, one has the following inequality for weights $k=3$ or $k=1$ (see~\cite{Duke88}*{Theorem 5} and~\cite{BM10}):
\begin{equation*}
  \rho_{j,k}(d)\ll \abs{r_j}^{5/4-\Msgn(d)/8} \Mch(\frac{\pi r_j}{2}) |d|^{-1/2+\delta}
\end{equation*}

We will require a version of the above inequality that is uniform over weights $k$. Fortunately, this is easy by using the raising operators and the normalization
given by equation~\eqref{raisingfactor}:

\begin{equation}\label{Duke-Iwaniec}
\begin{aligned}
\rho_{j,k}(d)\ll &\abs{r_j}^{5/4-\Msgn(d)/8} 
\Mch(\frac{\pi r_j}{2}) |d|^{-1/2+\delta}
\times \\
\times 
&\left(\prod_{m=0}^{\lfloor k/4\rfloor} |(k/2-1/2-m-ir_j)(k/2-1/2-m+ir_j)|^{1/2}
\right)^{-\Msgn(d)}
\end{aligned}
 \end{equation}
If $\psi_j$ does in fact come from a holomorphic form then the bound has in fact been worked out by Mao in an appendix to~\cite{BHM07}.

We shall more frequently use \eqref{Duke-Iwaniec} in the form stated below:

\begin{proposition}\label{Duke-Iwaniec:Clean}
There exists a real number $A>0$ such that for all forms $f_{j,k}$, $0<y<1$, and non-zero $d\in \BmZ$, we have:
\begin{equation}
W_{\Msgn(d)\frac{k}{2},ir}(y)\rho_{j,k}(d)\ll y^{1/2-\theta/2}(1+\abs{r_j}+\abs{k})^A\times |d|^{-1/2+\delta}
\end{equation}
\end{proposition}

\begin{proof} Let $\iota=\Msgn(d)=\pm 1$.
Assume that $f_{j,k}$ corresponds to a principal series. Then by Proposition \ref{prop:W-bound} and the Duke-Iwaniec bound \eqref{Duke-Iwaniec} we need only to show
that \[\abs{\Gamma(1/2+\iota k/2+ir_j)}\Mch(\frac{\pi r_j}{2})\times\left(\prod_{m=0}^{\lfloor k/4\rfloor} |(k/2-1/2-m+ir_j)|
\right)^{-\iota}\ll (1+\abs{r_j}+\abs{k})^A\]

Now, 
\begin{align*}
\abs{\Gamma(1/2+\iota k/2+ir_j)}\times\Mch(\frac{\pi r_j}{2})\left(\prod_{m=0}^{\lfloor k/4\rfloor}|(k/2-1/2-m+ir_j)|
\right)^{-\iota}&=\\
 \abs{\Gamma(1/2+\iota k/2 +ir_j)}\left(\frac{\abs{\Gamma(k/2-1/2+ir_j)}}{\abs{\Gamma(ir_j)}}\right)^{-\iota}\Mch(\frac{\pi r_j}{2})\\
\end{align*}

If $\iota=1$ (i.e. $d>0$) the result follows from $\Mch(it)\asymp e^{|t|}$ and $\Gamma(it)\asymp \abs{t}^{-1/2}e^{-\pi|t|/2}$ as $t\rightarrow\infty$. Else, if $\iota=-1$, we are reduced to showing that
\[\abs{\frac{\Gamma(k/2-1/2+ir_j)\Gamma(-k/2+1/2+ir_j)}{\Gamma(ir_j)^2}}\ll (1+\abs{r_j}+\abs{k})^A\] and the result follows from
Lemma~\ref{lem:gamma-product-shifted}. The cases where $f_{j,k}$ corresponds to a discrete or complementary series follow similarly.

\end{proof}

\subsection{Poincar\'e series}\label{sec:half:poincare} The Poincar\'e series of weight $k$ is defined by
\begin{equation}\label{def:P}
 P(z):=\sum_{\gamma\in \Gamma_\infty \SB \Gamma_0(4N)} f(\gamma z) \overline{J(\gamma,z)}^{2k}
\end{equation}
where $f(x+iy)=\Psi(y)e(dx)$. When $\Psi(y)=e^{-2\pi\abs{d}y}y^s$, we shall write $P_s$ for $P$. It is possible to continue $P_s$ analytically from the relation:
\begin{equation}
 [\Delta_{k}+s(1-s)]P_s =2\pi d(k-2s\Msgn(d))P_{s+1},
\end{equation}
following~\cite{Cohen-Sarnak} or \cite{Sarn84}. More precisely $P_s$ admits a meromorphic continuation to $\MRe s >\frac{1+\theta}{2}$ with a simple pole at $s=3/4$. We shall recall the explicit value of the residue in \S~\ref{sec:holo:main} (theta series).

\section{Holomorphic forms, part I}\label{sec:hol}
Here we shall follow the classical approach via Poincar\'e series, (see~\cites{Good83,Selb65}). Most recent papers where the method has been refined are~\cites{Sarn94,Harc03}. See also~\cite{cong:park:mich} for a survey. Recall that the essence of the method consists in forming $\langle P_s,f\overline{\theta}\rangle$, then on the one hand expanding spectrally with the Parseval relation and on the other hand unfolding the Poincar\'e series yielding a weighted Dirichlet series. The key ingredients involved are the analytic continuation of $P_s$ and a triple product estimate. The definition of $P_s$ has been given in \S\ref{sec:half:poincare}, where the weight is $k:=K-1/2$. 

\begin{remark}
We use the Poincar\'e series with $\Psi(y):=e^{-2\pi \abs{d}y} y^s$, because its Mellin transform is very explicit: its just a product of ratios of gamma functions. 
This has the affect of making the class of test functions we sum against more restricted as their Mellin transforms have exponential vertical decay. In section~\ref{sec:pf} 
we remove this restriction.
\end{remark}

\subsection{Introduction} We consider $f$ a newform in $S_K(\Gamma_0(N),\chi)$, with $\chi$ a character of level dividing $N$. 
\begin{equation}
 f(z)=y^{K/2}
 \sum^\infty_{n=1}
 n^{\frac{K-1}{2}}a(n)e(nz),
\quad z=x+iy.
\end{equation}

Given a sufficiently nice, smooth function $g$ on the reals, We wish to understand the following sum:
\begin{equation}
 \sum_{n\in\BmZ}
 a(n^2+d)
 g(\frac{2n^2+d}{Y})
\end{equation}

\subsection{Dirichlet series} The inner product $\langle,\rangle$ and the theta series are as in~\S\ref{sec:half:intro} and the Poincar\'e series is as in~\S\ref{sec:half:poincare}. We also set $z=x+iy$. Unfolding the Poincar\'e series one obtains (see also~\cite{Sarn84}*{(2.14)}):
\begin{equation}\label{PsfT}
\begin{aligned}
 \langle P_s,\overline{f} \theta \rangle&=
 \int_{\Gamma_\infty\SB \FmH}
 \overline{f(z)}\theta(z) e(dx) e^{-2\pi \abs{d}y}
 y^{s} \frac{dxdy}{y^2}\\
 &= \int^\infty_0 \int^1_0
 \Bigl(\displaystyle\sum_{m=0}^{\infty} m^{(K-1)/2} a(m) \overline{e(mz)} \Bigr)
 \Bigl(\displaystyle\sum_{n=-\infty}^{\infty} e(n^2 z)\Bigr)
 e(dx) e^{-2\pi \abs{d}y}
 y^{\frac{K}{2}+\frac{1}{4}+s} \frac{dxdy}{y^2}\\
 &=\frac{%
 \Gamma(\frac{K}{2}-\frac{3}{4}+s)}
 {%
 (2\pi)^{\frac{K}{2}-\frac{3}{4}+s}}
 \sum_{n^2+d>0}
 \frac{a(n^2+d)
 \abs{n^2+d}^{\frac{K-1}{2}}}
 {(2n^2+d + \abs{d})^{\frac{K}{2}-\frac{3}{4}+s}}.
\end{aligned}
\end{equation}
Observe that on the second line all terms are zero except when $m=n^2+d$.

In the next subsections we shall prove that the inner product is holomorphic for $\MRe s > \Mdemi+\theta$ with a possible simple pole at $s=3/4$ and control the growth on vertical lines uniformly in $d$. 

\subsection{Spectral expansion}
The next step is to expand spectrally the inner product via
Parseval's relation:
\begin{equation}
 \langle P_s,\overline{f} \theta\rangle =\sum_j
 \langle P_s,f_j\rangle \langle f_j,\overline{f} \theta\rangle 
 \pluscont
\end{equation}
We have set $f_j:=f_{j,k}$ for the present section~\ref{sec:hol} and recall that $k=K-\Mdemi$. The sum is over an orthonormal basis of Maass wave
forms of weight $k$ and nebentypus $\chi$ plus the continuous spectrum which is not displayed here, but whose contribution can be bounded in the same way as for 
the Maass forms.

\subsection{The Mellin transform of Whittaker function} Unfolding again we have (see \S\ref{sec:half:intro} for the expansion of weight $k$ automorphic forms and we have set $\rho_{j}:=\rho_{j,k}$):
\begin{equation}\label{Psfj}
\begin{aligned}
 \langle P_s,f_j\rangle 
 &=
 \int_{\Gamma_\infty\SB \FmH}
 \overline{f_j(z)}
 e^{-2\pi \abs{d}y}y^s e(dx) \frac{dxdy}{y^2}\\
 &=\overline{\rho_j(d)}
 \int^\infty_0
W_{\Msgn(d)k/2,\overline{ir_j}}
 (4\pi\abs{d}y)
  e^{-2\pi \abs{d}y}
  y^{s-1} \frac{dy}{y}\\
 &=\overline{\rho_j(d)}
 (4\pi \abs{d})^{1-s}
 \int^\infty_0
 W_{\Msgn(d)k/2,\overline{ir_j}}(y)
  e^{-y/2}
  y^{s-1} \frac{dy}{y}\\
 &=\overline{\rho_j(d)}
 (4\pi \abs{d})^{1-s}
\times
\frac
{\Gamma(s+ir_j-\Mdemi)\Gamma(s-ir_j-\Mdemi)}
{\Gamma(s-\Msgn(d)\frac{k}{2})}.
\end{aligned}
\end{equation}
In the last identity we have used~\cite{GR}*{(7.621-11)}, see also~\eqref{W-mellin} in the section \ref{sec:appendix}. We have also used the fact that $r_j$ is 
real or purely imaginary so that $\{ir_j,-ir_j\}=\{\overline{ir_j},-\overline{ir_j}\}$.

\subsection{Triple product estimate}
We shall give an estimate for $\langle f_j,\overline{f}\theta\rangle$ in terms of the eigenvalue $r_j$ as $r_j\rightarrow\infty$, with the correct exponential decay. Sarnak gave a very general such estimate in~\cite{Sarn94}, but it doesn't quite apply here (for instance, $\theta(z)$ is not a cusp-form). In our case, it turns out to be easier to prove this estimate directly. The main observation is that $\theta(z)$ is the residue of the unique pole of $E_{1/2}(s,z)$ at $s=3/4,$ so we can study the triple product 
$\langle f_j,\overline{f}E_{1/2}(s,z)\rangle$ as a function of $s$, and mimic standard L-function methods to get the desired estimate.

Namely, expanding for $\MRe(s)>K/2$ gives:

\begin{equation}\label{fjfE}
\begin{aligned}
R(s)&:=\langle f_j,\overline{f}E_{1/2}(s,z)\rangle
\\
&=
\int_{\Gamma_\infty\SB \FmH}
 f_j(z)f(z)y^s\frac{dxdy}{y^2}\\
&=\displaystyle\sum_{n=1}^{\infty}n^{\frac{K-1}{2}}a(n)\rho_j(n)\int_0^{\infty}W_{\Msgn(d)k/2,ir_j}(4\pi ny)e^{-2n\pi ny}y^{s+\frac{K}{2}-1}\frac{dy}{y}\\
&=\displaystyle\sum_{n=1}^{\infty}\frac{n^{\frac{K-1}{2}}a(n)\rho_j(n)}{(4\pi n)^{s+\frac{K}{2}-1}}\int_0^{\infty}W_{\Msgn(d)k/2,ir_j}(y)e^{-y/2}y^{s+\frac{K}{2}-1}\frac{dy}{y}\\
&= \frac{\Gamma(\frac{1}{2}+s+ir_j)\Gamma(\frac{1}{2}+s-ir_j)}{(4\pi)^{s+\frac{K}{2}-1}\Gamma(1+s-\Msgn(d)k/2)}\displaystyle\sum_{n=1}^{\infty}\frac{a(n)\rho_j(n)}{n^{s-\frac12}} \\
\end{aligned}
\end{equation}

Now, using the bound (see \eqref{Duke-Iwaniec}) $\rho_j(n)\ll |r_j|^ke^{\pi r_j/2}$, we see that $R(s)$ is uniformly bounded by $|r_j|^{3k}e^{-\abs{\abs{\MIm s}-\pi r_j}/2}$ on $k+1>\MRe(s)>k$. Since $R(s)$ satisfies a functional equation inherited from the functional equation for $E_{1/2}(s,z)$, we can bound it on $-k<\MRe(s)<1-k$ as well. Now using Phragmen-Lindelof, we get that
\begin{equation}\label{boundvertical}
\langle f_j,\overline{f}\theta\rangle \ll \abs{r_j}^k e^{-\pi r_j/2}
\end{equation}
as desired.

\subsection{Isolating the error term}

Going back to our expansion, we write it as
\begin{equation}
 \langle P_s,\overline{f} \theta\rangle =
 \sum_{f_j\in RES} \langle P_s,f_j\rangle \langle f_j,\overline{f} \theta\rangle
 +
 \sum_{f_j\not\in RES} \langle P_s,f_j\rangle \langle f_j,\overline{f} \theta\rangle 
 =S_0+S_1
\end{equation}

The second summand $S_1$ will be the error term which we deal with now. It's a sum over eigenfunctions not coming from the residual spectrum, and so by bounds towards Ramanujan each summand is holomorphic in $\MRe s > \frac{1+\theta}{2}$. Moreover, by Weyl's law and Stirling's formula we have the bound
\begin{equation}
\begin{aligned}
|S_1|&\ll_{k,\epsilon}
\sum_{f_j\not\in RES} \overline{\rho_j(d)}
 (4\pi \abs{d})^{1-s}
\times\frac{\Gamma(s+ir_j-\Mdemi)\Gamma(s-ir_j-\Mdemi)}
{\Gamma(s-\Msgn(d)\frac{k}{2})} \times |r_j|^ke^{-\pi r_j/2}\\
&\ll_{k,\epsilon}\sum_{f_j\not\in RES} \abs{d}^{1/2+\delta-s+\epsilon}|r_j|^k\times\frac{\Gamma(s+ir_j-\Mdemi)\Gamma(s-ir_j-\Mdemi)}
{\Gamma(s-\Msgn(d)\frac{k}{2})}\\
&\ll_{k,\epsilon} \abs{d}^{1/2-s+\delta+\epsilon}.
\end{aligned}
\end{equation}
We have used equations \eqref{boundvertical} and \eqref{Duke-Iwaniec}. The bound is uniform inside the critical strip and with $\frac{1+\theta}{2}+\epsilon \leq \MRe s \leq 3$. 

This gives us the promised error estimate, as long as we can identify the main term from the sum $S_0$ over the residual spectrum. We proceed to do this now.

\subsection{Main Term}\label{sec:holo:main} 
In this subsection we deal with the sum over the residual spectrum.  Recall that we are summing over forms of weight $k=K-1/2$. By the theory of raising operators, 
all the spectrum comes from weight $1/2$ if $K$ is odd, and from weight $3/2$ if $K$ is even (see \S\ref{sec:half:residual}). By~\cite{Duke88}, there is no 
residual spectrum of weight $3/2$, so we restrict to the case of $K$ odd. We first consider the case of $N$ square-free. Then by Serre-Stark~\cite{Se-St}, there 
is only residual spectrum of weight $1/2$ if $\chi=\chi_{4N}$, and in that case it's 1-dimensional and spanned by the theta function:
\begin{equation}
\theta_N(z)=y^{\frac{1}{4}}\displaystyle\sum_{n\in \BmZ} e(Nn^2z).
\end{equation}
We care about $u(z)$, which is $\theta_N(z)$ raised to level $k$, and normalized to be unitary. The $d$'th Fourier coefficient of $u(z)$ can be computed explicitly, and is equal to 0 unless $d\in N\cdot\mathbb{Z}^2,$ in which case:

\[\hat{u}(d)=\frac{1}{||\theta_N||_2}\cdot\Bigl(\displaystyle\prod_{i=1}^{\frac{K-1}{2}} (i)(i-1/2)\Bigr)^{-1/2}\]

Now we need to compute $\langle u,\overline{f}\theta\rangle$. To do this we again use that $u(z)$ is a multiple of 
the residue at $s=3/4$ of the weight $k$ Eisenstein series at level $N$, which we called $E_{N,k}(s,z)$. Using equation \eqref{eisensteinresidue} we deduce
that \[Res_{s=3/4}E_{N,k}(s,z)=\frac{u(z)}{||\theta_N||_2}\cdot\Bigl(\displaystyle\prod_{i=1}^{\frac{K-1}{2}} (i)(i-1/2)\Bigr)^{1/2}\]

 We compute $\langle E_{N,k}(s,*),\overline{f}\theta\rangle$ by unfolding the Eisenstein series, and then take the residue at $s=3/4$. The expansion is:

\begin{align*}
\langle E_{N,k}(s,*),\overline{f}\theta \rangle  
&=
\sum_{n\in \BmZ} a(n^2)\int_0^{\infty}e^{-2\pi n^2y}y^{s-3/4+K/2}\frac{dy}{y}\\
&=
 \frac{\Gamma(s-3/4+K/2)}{(2\pi)^{s-3/4+K/2}}\sum_{n\in \BmZ}\frac{a(n^2)}{n^{2s-1/2}}.\\
\end{align*}

We evaluate the Dirichlet series $\sum_{n}\frac{a(n^2)}{n^{2s-1/2}}$ to be $\zeta(4s-1)^{-1}L(2s-1/2,\Msym^2f)$, and so taking the residue at $s=3/4$ gives
$2\zeta^{-1}(2)Res_{s=1}L(s,\Msym^2f)\frac{\Gamma(K/2)}{(2\pi)^{K/2}}.$ Putting things together, we get that
\[\langle f\theta,u(z)\rangle \hat{u}(d) =2\displaystyle\prod_{i=1}^{\frac{K-1}{2}}\Bigl(i(i-1/2)\Bigr)^{-1}\Gamma(\frac{K}{2})(2\pi)^{K/2}\zeta^{-1}(2)Res_{s=1}L(s,\Msym^2f)\]

We now briefly consider the case of general level $N\ge 1$. In this case, the residual spectrum is furnished by a finite set of theta functions $\theta_{\psi,t}$, where $tL^2 \mid N$, and a $\psi(n)$ is a Dirichlet character $\psi:(\mathbb{Z}/L)^*\rightarrow\mathbb{C}^*$ such that $\chi_t\psi=\chi.$ Since they all arise as residues of Eisenstein series, the main term can in principle be computed as above, though we do not do so. 

With the above we see that the main term vanishes unless the following three conditions are satisfied:
\begin{enumerate}

\item $f(z)$ is a dihedral form;

\item $d$ is positive and divides $N$;

\item If $d=d'c^2$ where $d'$ is square-free, and $L$ is the conductor of $\chi\chi_d^{-1}$, then $d'L^2$ must divide $N$.

\end{enumerate}

\subsection{Summing up}
We write down the exact result in the case of $d>0$. In that case, we have

\[\langle P_s,\overline{f} \theta\rangle = \frac{%
 \Gamma(\frac{K}{2}-\frac{3}{4}+s)}
 {%
 (4\pi)^{\frac{K}{2}-\frac{3}{4}+s}}
 \sum_{n\in\mathbb{Z}}
 \frac{a(n^2+d)
 \abs{n^2+d}^{\frac{K-1}{2}}}
 {(2n^2+2d)^{\frac{K}{2}-\frac{3}{4}+s}} = \frac{%
 \Gamma(\frac{K}{2}-\frac{3}{4}+s)}
 {%
 2^{\frac{K}{2}-\frac{3}{4}+s}\pi^{\frac{K}{2}-\frac{3}{4}+s}}\sum_{n\in\mathbb{Z}}
 \frac{a(n^2+d)}
 {(n^2+d)^{s-\frac{1}{4}}}\]

Now let $g(x)$ be a smooth function on $\mathbb{R}_{+},$ such that the Mellin transform $\tilde{g}(s)$ decays sufficiently quickly on vertical strips (in particular, faster than $\Gamma(s)$). 
 Then the above estimates and a standard argument with shifting lines of integration gives, for $Y>d$,
 \[\displaystyle\sum_{n\in\mathbb{Z}} a(n^2+d)g\left(\frac{n^2+d}{Y}\right) = M_{f,d}\sqrt{Y}\tilde{g}\left(\frac{1}{2}\right) + 
O_{\epsilon}\left(\frac{Y^{\frac{\theta}{2} + \frac{1}{4}-\epsilon}}{d^{\frac{\theta}{2}-\delta+\epsilon}}\right)\]
where
\[M_{f,d}:= \begin{cases}
\frac{2^{K/2}(4\pi)^{1/4}\zeta^{-1}(2)}{\Gamma(3/4-k/2)}Res_{s=1}L(s,\Msym^2f)
\displaystyle\prod_{i=1}^{\frac{K-1}{2}}\Bigl(i(i-1/2)\Bigr)^{-1}& d\in N\Z^2, \chi=\chi_{4N}, K\notin 2\Z\\
0 & else\\
\end{cases}\]

\section{Holomorphic forms, part II}\label{sec:sketch}

In the next section~\ref{sec:pf} we will present a proof of our theorem that works uniformly for holomorphic and Maass forms. The goal of this section is to illustrate the main ideas of that proof without much of the technical difficulties. As such, in this section we prove our main theorem for holomorphic forms 
working with Maass forms in a fixed weight on the upper half plane. The problem with working in a fixed weight in the upper half-plane rather than on the metaplectic group is this restricts the test functions one can sum against. As such, we shall only be summing against a very restricted class of test functions. Nonetheless, the main ideas remain the same.

We proceed with the proof.
Fix $f(z)\in S_K(N,\chi)$, a holomorphic form of integral weight $K$. Assume $d > 0$ for simplicity.

\begin{equation}
f(z) = y^{\frac{K}{2}}\displaystyle\sum_{n>0} n^{\frac{K-1}{2}}a(n)e(nz)
\end{equation}

We again consider $f\bar{\theta}\in S_{K-\frac12}(N,\chi)$ and compute the $d$'th Fourier coefficient:

\begin{equation}\label{sumholoII}
\int_0^1 f(x+iy)\bar{\theta}(x+iy)e(-dx)dx = y^{\frac{K}{2}+\frac14}\displaystyle\sum_{n\in\Z}(n^2+d)^{\frac{K-1}{2}}a(n^2+d)e^{-(2n^2 + d)y}
\end{equation}

We will take $y=\frac{1}{Y}$ where Y is some large number and so \eqref{sumholoII} is the sum we wish to bound. Now, we expand $f\bar{\theta}$ spectrally in
$S_{K-\frac12}(N,\chi)$, singling out the terms coming from the residual spectrum

\begin{equation}\label{spectralholoII}
f\bar{\theta} = \sum_{\tau\in RES}c_\tau\psi_\tau + \sum_{\tau\not\in RES} c_\tau\psi_\tau
\end{equation}

where the sum ranges over distinct Maass forms of weight $K- \frac{1}{2}$, and $\psi_\tau$ is normalized so that $||\psi_\tau||=1$. The sum in 
\eqref{spectralholoII} converges absolutely and uniformly, so equating $d$'th Fourier coefficients we get the identity:

\begin{equation}\label{coeffholoII}
y^{\frac{K}{2}+\frac14}\displaystyle\sum_{n\in\Z}(n^2+d)^{\frac{K-1}{2}}a(n^2+d)e^{-\frac{(2n^2 + d)}{Y}} = \displaystyle\sum_\tau c_\tau\rho_\tau(d) W_{\frac{K}{2}-\frac14,ir_\tau}(\frac{d}{Y})
\end{equation}

As we have seen before, the finitely many terms coming from the residual spectrum will constitute the main term, as worked out in \S\ref{sec:holo:main}. Our goal therefore becomes to bound :
\begin{equation}
 \displaystyle\sum_{\tau\not\in RES} \rho_\tau(d)c_\tau W_{\frac{k-1}{4},r_\tau}(\frac{d}{Y})
\end{equation}
uniformly in $d$ and $Y$.

For a fixed representation $\tau$ outside the residual spectrum, the bound we wish to get is easy. Namely, $|\rho_\tau(d)|\ll_{\tau} d^{-1/2+\delta}$ and by the asymptotics of the Whittaker function near $0$ (See Prop. \ref{prop:W-bound}), combined with the bound towards Selberg's eigenvalue conjecture, $|W_{\frac{K}{2}-\frac14,r_\tau}(y)|\ll_{\tau} y^{\abs{\MRe(ir_\tau)}+\frac12}\leq y^{\frac{1}{2}-\frac{\theta}2}$. Multiplying, this gives the bound  
\begin{equation}
\abs{c_\tau\rho_\tau(d) W_{\frac{K}{2}-\frac14,ir_\tau}(\frac{d}{Y})}\ll_\tau d^{-1/2+\delta}
\Bigl(\frac{d}{Y}\Bigr)^{\frac{1-\theta}{2}}
\end{equation}

What we need is a version of \eqref{spectralholoII} that is uniform in $\tau$, or more precisely  in $r_\tau$. Since both the asymptotics for the Whittaker function and the Ramanujan bound are already uniform up to polynomial dependence, what we really need is good control in the $c_\tau$. To accomplish this, we introduce a light version of Sobolev norms. Namely, for a function $\psi(z)$ on $S_k(N,\chi)$, we define the $m'th$ Sobolev norm to be

$$||\psi||_m = ||\Delta_k^{(m)}\psi||.$$

That is, we apply the weight-$k$ Laplacian $m$ times, and then take usual $L^2$ norm. 

It is easy to see (via explicit Fourier expansion, for example) that for $A\geq 0, \Delta_k^{(A)}(f\bar{\theta})$ decays exponentially at the cusps, and thus the $A$'th Sobolev norm $||f\bar{\theta}(z)||_A$ is finite.
Now using the spectral expansion \eqref{spectralholoII} and applying the Laplacian $A$ times we get
\[\Delta_k^{(A)}(f\bar{\theta}) = \sum_{\tau}c_\tau(\frac14+|r_\tau|^2)^A\psi_\tau\]
Taking $L^2$ norms and using Parseval's relation we get the bound
\[\sum_\tau |r_\tau|^{2A} |c_\tau|^2\leq ||f\bar{\theta}||_A\]
and the immediate corollary:
\begin{equation}\label{ctau}
\abs{c_\tau}\ll_C |r_\tau|^C
\end{equation}
for any real number C.

Equation~\eqref{ctau} allows us as  much polynomial control as we want, and so we can finish the argument with an application of Cauchy-Schwarz. By Proposition~\ref{prop:W-bound} regarding the asymptotics of the Whittaker function and the Duke-Iwaniec bound~\eqref{Duke-Iwaniec} for Fourier coefficients there is a $B>0$ such that 
\begin{equation}\label{uniformholoII}
\abs{\rho_\tau(d) W_{\frac{K}{2}-\frac14,r_\tau}(\frac{d}{Y})}\ll d^{-1/2+\delta}\Bigl(\frac{d}{Y}\Bigr)^{\frac{1-\theta}{2}}|r_\tau|^B
\end{equation}
 uniformly in $\tau$. Combining the above with~\eqref{ctau} for $C=-B-3$ and Weyl's law, we get the desired error estimate:
\begin{equation}
\begin{aligned}
\abs{ \displaystyle\sum_{\tau\not\in RES} \rho_\tau(d)c_\tau W_{\frac{k-1}{4},r_\tau}(\frac{d}{Y})}
&\ll d^{-1/2+\delta}\Bigl(\frac{d}{Y}\Bigr)^{\frac{1-\theta}{2}}\displaystyle\sum_{\tau\not\in RES}\abs{c_\tau}|r_\tau |^B\\
&\ll d^{-1/2+\delta}\Bigl(\frac{d}{Y}\Bigr)^{\frac{1-\theta}{2}}\displaystyle\sum_{\tau\not\in RES}|r_\tau |^{-3}\\
&\ll d^{-1/2+\delta}\Bigl(\frac{d}{Y}\Bigr)^{\frac{1-\theta}{2}}
\end{aligned}
\end{equation}

For convenience, we plug in \eqref{uniformholoII}  back into \eqref{coeffholoII} to get
\[
Y^{\frac{-K}{2}-\frac14}\displaystyle\sum_{n\in\Z}(n^2+d)^{\frac{K-1}{2}}a(n^2+d)e^{-\frac{(2n^2 + d)}{Y}} = \sum_{\tau\in RES}c_\tau\rho_\tau(d)W_{\frac{K}{2}-\frac14,\lambda_\tau}(\frac{d}{Y}) + O(d^{-1/2+\delta}\Bigl(\frac{d}{Y}\Bigr)^{\frac{1-\theta}{2}})\]

The main term can now be more explicitly computed just as in section \ref{sec:hol}.

\section{Proof of Theorem~\ref{th:main}}\label{sec:pf}
We recall that for Maass forms the classical approach via Poincar\'e series to the shifted convolution problem fails to produce a suitable estimate because there are missing harmonics. For forms of half-integral weight the situation is even worse because more integrals are hypergeometric functions that cannot be expressed as a product of Bessel functions. Although this problem is purely local (archimedean) it is a delicate one.

A good solution is to use the framework of representation theory. In the context of split shifted convolution problems this is achieved in Blomer-Harcos~\cite{BH08} following the works of Bernstein-Reznikov and Venkatesh on Sobolev norms in the framework of representation theory. The main idea is that the missing harmonics are to be found in the higher weight vectors in the automorphic representation. Retrospectively the classical approach with Poincar\'e series only bears the new vector of the representation which is not flexible enough. Our proof of Theorem~\ref{th:main} is based on that idea as well. We shall try to use notation of~\cite{BH08} as closely as possible for convenience of the reader.

There is an important difference with the integral weight case that we would like to highlight. Seeing as how we shall have to work on the metaplectic group, all the estimates on Whittaker function will be gotten \Lquote{bare hands} without resorting to the Kirillov model. This is because Kirillov models for half-integral forms are different and we cannot use it in our context.

\subsection{Choice of local vector}
Following~\S\ref{sec:gl2:kirillov} we start out with picking an appropriate smooth vector $\phi$ for the $\GL_2$ automorphic representation $\pi$. Note that $\phi$ is not $K$-finite in general. We have the expansion
\begin{equation}
\phi(n(x)a(u)) = \sum_{n\neq0} \frac{a(n)}{\sqrt{n}}W_\phi(nu)e(nx)
,\quad 
x,u\in \BmR,\ u>0,
\end{equation}
where $W_{\phi}(y) = \int_0^1 \phi(n(x)a(y))e(-x)dx$ and $a(n):=a_\pi(n)$.

The Whittaker transform can be made to be any smooth function of compact support according to Proposition~\ref{prop:wtransform}. So we pick a smooth vector $\phi\in \pi$ such that 
\begin{equation}\label{choiceWphi}
W_\phi(y)= \exp(\frac{-2\pi d}{Y}) W(y)y^{1/2}e^y,\quad y>0. 
\end{equation}
Here $W$ is the function from~\eqref{eq:th:main} in Theorem~\ref{th:main} which we recall is smooth and compact support on $(1,2)$. Because of the assumptions on $W$ we have that $\pnorm{W^{(i)}_\phi}\ll 1$ for all $i\ge 1$, see also the remarks following~Theorem~\ref{th:main}. As recalled in~\S\ref{sec:gl2:kirillov}, we deduce that $\CmS_B \phi \ll 1$ for all $B$. The multiplicative constant may depend on $B$ only.

For simplicity we work with $G=\SL_2(\BmR)$ instead of the more general group $\SL^{\pm}_2(\BmR)$, where all the $\GL_2$ automorphic forms naturally live. This means that from now on we view $\phi$ as an element of $\lcusp$.

\subsection{Sobolev norms on the metaplectic cover} We define Sobolev norms on functions on the metaplectic cover $\Gamma \SB \widetilde G$ in the same way as for integral forms on $\GL(2)$. Namely the Lie algebra of $\widetilde G$ is identified with $\Fms\Fml_2(\BmR)$. Recall the basis formed by the matrices $L,H,R$. Given a smooth function on $\Gamma \SB \widetilde G$ and an integer $d\ge 0$ we define $\CmS_d \phi=\sum_{\FmD} \pnorm{\FmD\phi}$ where $\FmD$ ranges over all monomials in $H,L,R$ of degree at most $d$.

The only occurrence of these norms is in Plancherel relation~\eqref{plancherel} and Lemma~\ref{lem:sobbound} below.

\subsection{Spectral expansion on the metaplectic group}
We lift $\phi$ to the metaplectic group $\widetilde{G}$ in the obvious way. Since $\theta$ is a genuine  on $\widetilde{G}$ the product $\phi\overline{\theta}$ is a genuine function in $\lgen$. The next step is to expand spectrally that function according to the orthonormal basis from~\S\ref{sec:half:basis}.

Expanding, we arrive at
\begin{equation}\label{expansion-half}
\phi\overline{\theta} = \sum_{\tau} \psi_{\tau}
\pluscont,
\end{equation}
where $\tau$ corresponds to a certain genuine representation of eigenvalue $\frac{1}{4}+r_\tau^2$ and the vector $\psi_\tau$ is of weight $p_\tau$. In particular the form $\psi_\tau$ is always proportional to one element $f_{j,p}$ in the basis described in~\S\ref{sec:half:basis}.

Recall that because $\phi\overline{\theta}$ is genuine, the sum above is restricted to genuine representations $\tau$. Classically the expansion~\eqref{expansion-half} would involve only Maass forms of half-integral weight.

The sum converges in the Sobolev norm topology. More precisely the Plancherel formula and an iterative application of the Laplacian give
\begin{equation}\label{plancherel}
\sum_{\tau} (1+\abs{r_\tau}+\abs{p_\tau})^{2A}
\pnorm{\psi_{\tau}}^2 \ll \CmS_B(\phi\overline{\theta})^2 
\end{equation}
and a similar bound for the continuous spectrum. Here $B>0$ is some absolute constant that depends only on $A$.

\subsection{Uniform estimate}
\begin{lemma}\label{lem:sobbound}
For all $B>0$, $\CmS_B(\phi\overline{\theta})\ll 1$. The multiplicative constant depends only on $B$.
\end{lemma}
\begin{proof} This will follows from the bound $\pnorm{\phi}_B\ll 1$. Care has to be taken because $\theta$ is not a cusp form. This is resolved by introducing a further argument and controlling the growth towards the cusps.

Let $\Ht:\widetilde\Gamma \SB \widetilde G \to \BmR_{\ge 1}$ be a height function as in Michel-Venkatesh~\cite{MV10}. In the familiar case of $\SL_2(\BmZ)\SB \SL_2(\BmR)$, the function $\Ht$ represents the inverse of the shortest vector, in the standard fundamental domain it is given by the ordinate.
According to assertion S3b in~\cite{MV10}*{\S2.4.3}, since $\phi$ is a cusp form, we have $\CmS_B(\Ht^B \phi)\ll 1$ uniformly.

To bound the Sobolev norm $\CmS_B(\phi\overline{\theta})$ we only need to consider the $L^2$-norm $\pnorm{\FmD_1 \phi \FmD_2 \overline{\theta}}_2$ for differential operators $\FmD_1,\FmD_2$ which are monomials in $L,H,R$ of bounded degree. Since 
\begin{equation}
\pnorm{\FmD_1 \phi \FmD_2 \overline{\theta}}_2 \le \pnorm{\Ht^C \FmD_1 \phi}_2 \pnorm{\Ht^{-C} \FmD_2 \theta}_\infty,
\end{equation}
we may choose $C$ large enough so that $\Ht^{-C}$ compensates for all polynomial growths in the derivatives of $\theta$. This concludes the proof.
\end{proof}

\subsection{Unipotent integral}
On the group $\widetilde{G}$ we may expand
\begin{equation}
\theta(n(x)a(y)) = y^{1/4} \sum_{n\geq 0} e^{-2\pi n^2y}e(n^2x),\quad 
x,y\in \BmR,\ y>0.        %Define T. Should be roughly e^(-n^2y) but might change when lifting theta
\end{equation}

We obtain therefore
\begin{equation}
I:=\int_0^1 \phi\overline{\theta} (n(x)a(\frac{1}{Y}))e(-dx)dx = Y^{-1/4} \sum_{n\geq 0}
\frac{a(d+n^2)}{\sqrt{d+n^2}}W_{\phi}(\frac{d+n^2}{Y})e^{-2\pi \frac{n^2}{Y}}
\end{equation}

Because of our choice of vector $\phi$ and more specifically because of~\eqref{choiceWphi} we obtain that
\begin{equation}\label{I-mainth}
I=Y^{-3/4} \sum_{n\geq 0}
a(d+n^2)W(\frac{d+n^2}{Y}).
\end{equation}
Thus $I$ is equal to $Y^{-3/4}$ times the sum we are looking for in the statement of Theorem~\ref{th:main}.

\subsection{Decomposition of \texorpdfstring{$I$}{I}}
We now insert~\eqref{expansion-half} in the integral $I$ above. For all $\tau$ we have
\begin{equation}\label{def:Wtau}
\int^1_0 \psi_\tau(n(x)a(\frac{1}{Y}))e(-dx)dx=\frac{\rho_\tau(d)}{\sqrt{d}} W_{\psi_\tau}(\frac{\abs{d}}{Y}).
\end{equation}
where $\rho_\tau(d)$ is as in \S\ref{sec:half:intro}. In particular the bound~\eqref{Duke-Iwaniec} applies as stated. The identity~\eqref{def:Wtau} could be taken as definition of $W_{\psi_\tau}$ for our purpose.

% \begin{equation}
% %\sum_{n\geq 0} \frac{a(D+n^2)}{\sqrt{D+n^2}}W_{\phi}(\frac{D+n^2}{Y})e^{-\frac{n^2}{Y}}
% I= 
% \sum_{\tau\in RES}
% \frac{\rho_\tau(D)}{\sqrt{D}}
% W_{\psi_{\tau}}(\frac{D}{Y}) + \sum_{\tau\notin RES}
% \end{equation}

The expansion~\eqref{expansion-half} also splits naturally into the residual
spectrum coming from the theta series which yields the main term, and the rest of the spectrum which will end up being an error
term. We denote this by
\begin{equation}
I=\Ires + \Ioff
\end{equation}
Compared to~\S\ref{sec:hol} we are now picking it up directly from
the spectral expansion.

\subsection{Non-residual spectrum} We shall first bound the non-residual spectrum $\Ioff$.
Recall that $\theta$ is the exponent towards Selberg eigenvalue conjecture.
The main inequality is 
\begin{lemma} Assume that $\tau$ corresponds to a principal or complementary series. Then,
\begin{equation}\label{main-W-sobolev}
\frac{\abs{W_{\psi_{\tau}}(u)}}
 {\abs{\Gamma(\Mdemi+p_\tau+ir_\tau)}}
 \ll_\epsilon u^{\frac{1}{2} - \frac{\theta}{2}-\epsilon} (1+\abs{r_\tau}+\abs{p_\tau})^A ||\psi_{\tau}||.
\end{equation}
Assume that $\tau$ corresponds to an holomorphic series. Then
\begin{equation}
\frac{\abs{W_{\psi_{\tau}}(u)}}
{\abs{\Gamma(\Mdemi+p_\tau-ir_\tau)\Gamma(\Mdemi+p_\tau+ir_\tau)}^{1/2}}
 \ll_\epsilon u^{\frac{1}{2} - \frac{\theta}{2}-\epsilon} (1+\abs{r_\tau}+\abs{p_\tau})^A ||\psi_{\tau}||.
\end{equation}
\end{lemma}

\begin{proof} Since $\psi_\tau$ is proportional to the normalized $f_{j,p}$, it is sufficient to establish~\eqref{main-W-sobolev} with $f_{j,p}$ in place of $\psi_\tau$. Then it follows from Proposition~\ref{prop:W-bound} and the fact that $\pnorm{f_{j,p}}=1$. In more details, if $\psi_\tau=\alpha f_{j,p}$ for some $\alpha$ then $W_{\psi_\tau}=\alpha W_{p,i r_j}$ and $\pnorm{\psi_\tau}=\abs{\alpha}$.
\end{proof}

Then we apply Duke-Iwaniec bound~\eqref{Duke-Iwaniec} for $\rho_\tau(d)$ in the form of Corollary~\ref{Duke-Iwaniec:Clean} and the Whittaker bounds above to obtain
\begin{equation}
\begin{aligned}
\Ioff &= \sum_{\tau} 
\frac{\rho_\tau(d)}{\sqrt{\abs{d}}}
W_{\psi_\tau}\bigl(\frac{\abs{d}}{Y}\bigr)
\\
&\ll \sum_{\tau} 
\bigl(\frac{\abs{d}}{Y}\bigr)^{1/2-\theta/2} 
\abs{d}^{-1/2+\delta}
(1+\abs{r_\tau}+\abs{p_\tau})^{A}
\pnorm{\psi_\tau}
\\
&\ll \bigl(\frac{\abs{d}}{Y}\bigr)^{1/2-\theta/2} 
\abs{d}^{-1/2+\delta}.
\end{aligned}
\end{equation}
The last inequality follows from Cauchy-Schwarz inequality and \eqref{plancherel}. 

Taking into account~\eqref{I-mainth}, the estimate for $\Ioff$ corresponds exactly to the remainder term in~\eqref{eq:th:main} in Theorem~\ref{th:main}.

\subsection{Residual spectrum} We are now concerned with the contribution from the residual spectrum $\Ires$. In view of the description of the residual spectrum in \S\ref{sec:half:residual} and~\eqref{def:Wtau} the term $\Ires$ is proportional to $Y^{-1/4}$. It is also linear in $W$ because all constructions in the proof are linear in $W$. Therefore we have
\begin{equation}
\Ires = I(W)M_{\pi,d} Y^{-1/4}
\end{equation}
for some constant $M_{\pi,d}$. This main term could be derived in the same way as in \S\ref{sec:hol}. This would involve more machinery on integral representations of $L$-functions and therefore we have settled for determining only when it is zero.

Now to conclude the proof of Theorem~\ref{th:main} we observe that the constant $M_{\pi,d}$ has to vanish in the following cases. When $d<0$, the Fourier coefficients $\rho_\tau(d)$ of the residual spectrum representations $\tau$ are identically zero. For $\psi_\tau$ from the residual spectrum to contribute non-trivially to~\eqref{expansion-half} it is necessary that the inner product $\langle f\bar \theta,\psi_\tau\rangle$ be non-zero. By the theory of the Shimura integral this can happen only when $L(s,\Msym^2 f)$ has a pole at $s=1$, namely when $f$ is dihedral.

\section{Bounds for Whittaker functions}\label{sec:appendix}

\subsection{Whittaker functions} Our definition of Whittaker functions is as in section 9.2 of~\cite{GR}. Let $p,r\in\BmC$. The Whittaker function satisfies the following differential equation
\begin{equation}
 \frac{d^2 W_{p,ir}}{dy^2}
 +(-\frac{1}{4}+\frac{p}{y}+\frac{1/4+r^2}{y^2})W_{p,ir}=0.
\end{equation}

The differential equation has a regular singularity at zero and an irregular singularity at infinity. Up to scalars, $y\mapsto W_{p,ir}(y)$ is the unique function that decays as $y\to \infty$ (exponentially). An important difficulty in the theory of Whittaker functions is the normalization of the scalar. There doesn't seem to be a canonical normalization in general although the integral representations~\eqref{W-mellin},\eqref{W-haenkel} and~\eqref{W-limiting} produce such. On the practical side this makes some formula differ from place to place in the literature, and on the theoretical side some more care has to be taken when working with Whittaker models attached to automorphic forms.

\subsection{Integral representations} In this subsection we briefly summarize the classical integral representations of $W_{p,ir}$ and how they relate to each other.

According to~\cite{GR}*{(7.621-11)} one has
\begin{equation}\label{W-mellin}
 \int^\infty_{0}
 W_{p,it}(y)e^{-y/2}y^{s}\frac{dy}{y}
 = \frac{\Gamma(\frac{1}{2}+s+it)\Gamma(\frac{1}{2}+s-it)}{\Gamma(1+s-p)}, 
 \quad \MRe s > -1.
\end{equation}

One has the following Hankel integral representation, see~\cite{book:WW}*{\S16.12}:
\begin{equation}\label{W-haenkel}
 W_{p,\nu}(y)=
 \Gamma(p+\Mdemi-\nu)
 e^{-y/2}y^p
\int_{\CmH} (-t)^{\nu-p-\Mdemi}(1+\frac{t}{y})^{p-\Mdemi+\nu}e^{-t}\frac{dt}{2i\pi}.
\end{equation}
The formula is valid for all $y\in \BmR^\times_+$ and $p,\nu\in\BmC$ with the assumption that $p+\Mdemi-\nu$ is not a negative integer. The reference~\cite{book:WW}*{\S16.12} has a modified formula in that case, which is not repeated here because we won't make use of that formula. Here $\CmH$ is Hankel's contour which surrounds the positive real axis and is such that the real number $-y$ lies \Lquote{outside}.

For $\MRe (\nu-p)+\Mdemi>0$ it is possible by a limiting argument to obtain the following~\cite{GR}*{9.222-2}:
\begin{equation}\label{W-limiting}
 W_{p,\nu}(z)=
 \frac{z^p e^{-z/2}}
 {\Gamma(\nu-p+\Mdemi)}
 \int^\infty_0 t^{\nu-p-\Mdemi}
 (1+\frac{t}{z})^{\nu+p-\Mdemi}
 e^{-t}dt.
\end{equation}

\subsection{Asymptotics as \texorpdfstring{$y$}{y} goes to zero}\label{sec:appendix:zero} In this subsection we recall briefly the asymptotic behavior of the Whittaker functions as $y\to 0$. The main purpose of this subsection is to have a consistency test for the numerical values of the constants in the computations. The Proposition~\ref{prop:W-bound} that we shall establish below is more precise because it is uniform in $p$ and $r$ as well.
 
From~\eqref{W-mellin} we deduce that
\begin{equation}\label{W-zero-asymp}
W_{p,ir}(y) \sim 
\frac{\Gamma(-2ir)}{\Gamma(\Mdemi-ir-p)}
y^{\Mdemi+ir}
+
\frac{\Gamma(2ir)}{\Gamma(\Mdemi+ir-p)}
y^{\Mdemi-ir},
\qtext{as $y\to 0$.}
\end{equation}
This is consistent with the power series expansion of confluent hypergeometric functions
 \begin{equation}
\Phi(a,b;y)=\sum^\infty_{n=0}
\frac{(a)_n y^n}{(b)_n n!}.
\end{equation}
 Indeed according to~\cite{GR}*{9.220} and~\cite{GR}*{9.210} we have
\begin{equation}
\begin{aligned}
W_{p,ir}(y) e^{y/2}=\frac{\Gamma(-2ir)}{\Gamma(\Mdemi-ir-p)}
y^{\Mdemi+ir}
&\Phi(\Mdemi+ir-p,1-2ir;y)\\
&+
\frac{\Gamma(2ir)}{\Gamma(\Mdemi+ir-p)}
y^{\Mdemi-ir}
\Phi(\Mdemi-ir-p,1-2ir;y).
\end{aligned}
\end{equation}

\subsection{Preliminary lemmas} In preparation for the proof of Proposition~\ref{prop:W-bound} we recall some elementary estimates concerning certain ratios of the Gamma function. We were not able to locate several of the claims in the literature so that we provide brief proofs for the sake of completness.

We start with an elementary fact which should be more widely known and will be used repeatedly.
\begin{lemma}\label{lem:gamma-decay} For all $\sigma\in \BmR$,
$
t\mapsto \abs{\Gamma(\sigma+it)}
$
is a decreasing function of $\abs{t}$.
\end{lemma}
\begin{proof}
 This follows for instance from the Weierstrass product formula
\begin{equation}
\Gamma(s)=\frac{e^{-\gamma s}}{s}
\prod^\infty_{n=1}
(1+\frac{s}{n})^{-1}
e^{s/n}.
\end{equation}
\end{proof}

Next we recall the Stirling formula: 
\begin{lemma}\label{lem:stirling}
Let $\epsilon>0$. Uniformly on $s$ with $\abs{\arg s}>\pi-\epsilon$ the following holds
\begin{equation}
\Gamma(s)=(\frac{2\pi}{s})^{1/2}
(\frac{s}{e})^s
(1+O(\abs{s}^{-1})).
\end{equation}
\end{lemma}

The following lemma will be useful when handling integrals on vertical lines in the Mellin inversion formulas.
\begin{lemma}\label{lem:gamma-infty} For all fixed $\delta>0$, $\Gamma(t+i\abs{t}^{1+\delta})$ is exponentially small as $t\to \pm \infty$.
\end{lemma}
\begin{proof}
As $t\to \infty$, the Stirling formula implies that
\begin{equation}
\abs{\Gamma(t+it^{1+\delta})}
\asymp t^{(1+\delta)(t-1/2)}
e^{-\pi t^{1+\delta}/2}
\ll e^{-t^{1+\delta'}},
\end{equation}
for all $\delta'<\delta$. When $t\to -\infty$, the order of magnitude is even smaller by the recursion $\Gamma(s+1)=s\Gamma(s)$.
\end{proof}

For $p\in \BmR$, let $\fpart{p}\in [0,1)$ denote the fractional part. For $z\in \BmC$, let $\pnorm{z}=\min\limits_{n\in\BmZ}\abs{z-n}$ be the distance to the nearest integer. 
\begin{lemma}\label{lem:gamma-product-shifted}
Let $\epsilon>0$ and a large integer $A\ge 1$ be given. Uniformly on $a,b \in \BmC$ and $p\in \BmR$ with $\abs{\MRe a},\abs{\MRe b}<A$, $\Im a=\Im b$ and $\pnorm{a+p},\pnorm{b-p}\ge \epsilon$, the following holds
\begin{equation}
h^{-2A}
\ll
\frac{\abs{\Gamma(a+p)\Gamma(b-p)}}
{\abs{\Gamma(a+\fpart{p})\Gamma(b-\fpart{p})}} \ll h^{2A}.
\end{equation}
Here $h=\abs{p}+\abs{\Im a}+1$ and the implied multiplicative constants may depend on $\epsilon,A$.
\end{lemma}

The assumption that $a$ and $b$ have the same imaginary part is necessary in the above lemma. In the proof this is used for the inequality~\eqref{pf:gamma-product:3}. The exponent $2A$ is far from optimal although sufficient for our purpose.

\begin{proof} Exchanging $a$ and $b$ and turning $p$ into $-p$, we may assume without loss of generality that $p>0$. Exchanging $a$ into $1-b$ and $b$ into $1-a$, it is sufficient to prove the upper bound only because
\footnote{One could also compare directly the two ratios.}
of Euler's reflection formula $\Gamma(s)\Gamma(1-s)=\dfrac{\pi}{\sin(\pi s)}$.

Without loss of generality we may assume that $A\ge 1$ is a positive integer.
It is not difficult to see that
\begin{equation}
\frac{\abs{\Gamma(b-p)}}
{\abs{\Gamma(b-\fpart{p})}}\ll_{\epsilon,A}
\frac{\abs{\Gamma(b-2A-p)}}
{\abs{\Gamma(b-2A-\fpart{p})}} h^{2A}.
\end{equation}

Let $n=p-\fpart{p}\in \BmN$. We have
\begin{equation}\label{pf:gamma-product:3}
\frac{\abs{\Gamma(a+p)\Gamma(b-2A-p)}}
{\abs{\Gamma(a+\fpart{p})\Gamma(b-2A-\fpart{p})}}
=\prod^n_{i=1}
\frac{\abs{a+\fpart{p}-1+i}}{\abs{b-2A-\fpart{p}-i}} \le 1
\end{equation}
This is because in absolute values, the real part of the numerator is smaller than the real part of the denominator and we recall that $\Im a=\Im b$.
\end{proof}

\subsection{Bounds for Whittaker functions}\label{sec:appendix:pf} We provide now a proof of Proposition~\ref{prop:W-bound}. 

We start with the Mellin inversion of~\eqref{W-mellin} which reads
\begin{equation}\label{pf:eq-Mellin}
W_{p,\nu}(y)e^{-y/2}=
\int_{\MRe s =\sigma} 
\frac{\Gamma(\Mdemi+s+\nu)\Gamma(\Mdemi+s-\nu)}
{\Gamma(1+s-p)}
y^{-s}
\frac{ds}{2i\pi},
\end{equation}
where $\sigma$ is sufficiently large. In all three cases (i-iii) in Proposition~\ref{prop:W-bound} we move the line of integration to $\sigma=0$. We are reduced to controlling the residues on the one hand and the integral on $\sigma=0$ on the other hand. For the integral we shall distinguish between those $s$ with small and large imaginary part. In the sequel we use $B>0$ to denote a large constant that may vary from line to line.

\subsubsection{Residues} Before going into the proof we remark that we may cross poles at $s=\Mdemi\pm \nu-\BmN$. The residues at those points give the two terms in the asymptotic~\eqref{W-zero-asymp}. Also we recall the normalizing factors $\Gamma(\Mdemi+p+ir)$, $\Gamma(\Mdemi+p)$ and $\abs{\Gamma(\Mdemi+p-\nu)\Gamma(\Mdemi+p+\nu)}^{1/2}$ in the left-hand side in cases (i), (ii) and (iii) respectively.

(i) The residue at $s=-\Mdemi-ir$ accounts for
\begin{equation}
\frac{\Gamma(2ir)}{\Gamma(\Mdemi-p-ir)\Gamma(\Mdemi+p+ir)}y^{\Mdemi+ir}
\ll (\abs{p}+\abs{r}+1)^B y^{1/2}
\end{equation}
The upper bound follows from Lemma~\ref{lem:gamma-product-shifted} and Stirling formula. The other residue at $s=-\Mdemi+ir$ is similar. The estimate is admissible compared to the right-hand side of~\eqref{eq:W-bound:i}.

(ii) The residue at $s=-\Mdemi-\nu$ yields a lower order term. The larger term will arise from the residue at $s=-\Mdemi+\nu$. It accounts for
\begin{equation}
\frac{\Gamma(2\nu)}{\Gamma(\Mdemi+\nu-p)\Gamma(\Mdemi+p)} y^{1/2-\nu}
\ll
(\abs{p}+1)^B y^{1/2-\nu}
\end{equation}
Again this estimate is admissible compared to the right-hand side of~\eqref{eq:W-bound:ii}.

(iii) We observe that because of the assumption that $p-\nu-\Mdemi\in\BmN$, the ratio $s\mapsto \frac{\Gamma(\Mdemi+s-\nu)}{\Gamma(1+s-p)}$ has no pole at all, and actually is a polynomial. The first pole of $\Gamma(\Mdemi+s+\nu)$ is when $s=-\Mdemi-\nu$. Since $\nu>-\Mdemi$ there is no residue while moving the line of integration to $\sigma =0$.

\subsubsection{Mellin integrals.} 
It remains to estimate the integral~\eqref{pf:eq-Mellin} when $\sigma=0$.

(i) We want to bound the following ratio (put $s=it$ and recall that $\nu=ir$):
\begin{equation}
\frac{\Gamma(\Mdemi+it+ir)\Gamma(\Mdemi+it-ir)}
{\Gamma(\Mdemi+ir+p)\Gamma(1+it-p)}.
\end{equation}
Precisely we shall exhibit a fast decay when $t$ goes to infinity with a polynomial control in $p$ and $r$.

Because of the decay of the Gamma function on vertical lines (Lemma~\ref{lem:gamma-decay}), we may modify the denominator so that the two Gamma factors are evaluated at the same imaginary part $\max(\abs{t},\abs{r})$. Then we are in position to apply Lemma~\ref{lem:gamma-product-shifted} which enables us to replace $p$ by $\fpart{p}$. After this is done, we apply Stirling's formula (Lemma~\ref{lem:stirling}).

(ii) We need to bound the following ratio
\begin{equation}
\frac{\Gamma(\Mdemi+it+\nu)\Gamma(\Mdemi+it-\nu)}
{\Gamma(1+it-p)\Gamma(\Mdemi+p)}.
\end{equation}
The proof is entirely similar.

(iii) We need to bound the ratio
\begin{equation}\label{pf:iii}
\frac{\Gamma(\Mdemi+\nu+it)\Gamma(\Mdemi-\nu+it)}
{\abs{\Gamma(\Mdemi+p-\nu)\Gamma(\Mdemi+p+\nu)}^{1/2}\Gamma(1+it-p)}.
\end{equation}
For the denominator we first observe that $\abs{\Gamma(\Mdemi+p-\nu)\Gamma(\Mdemi+p+\nu)}^{1/2}$ is at least $\Gamma(p)$. It remains
\begin{equation}
\ll 
\frac{\Gamma(it+p)}{\Gamma(p)}
\frac{\Gamma(\Mdemi+\nu+it)\Gamma(\Mdemi-\nu+it)}
{\abs{\Gamma(it+p)\Gamma(1+it-p)}}
\end{equation}
We apply here Lemma~\ref{lem:gamma-decay} to bound $\frac{\Gamma(it+p)}{\Gamma(p)}$ by one. Also because of Lemma~\ref{lem:stirling} and Lemma~\ref{lem:gamma-infty} the ratio $\frac{\Gamma(it+p)}{\Gamma(p)}$ decays exponentially as $|t|>p^{1+\delta}$.

Since $\nu\in \BmR$, we may apply~Lemma~\ref{lem:gamma-product-shifted} to the product of Gamma functions on the numerator. This enables to replace $\nu$ by $\{\nu\}$. We  apply Lemma~\ref{lem:gamma-product-shifted} to the denominator as well replacing $p$ by $\{p\}$. We conclude with the Stirling formula that the second term is bounded uniformly by a polynomial in $p$ and $\nu$. This concludes the proof of Proposition~\ref{prop:W-bound}. \qed

%There is a small remark on the proof of (iii) of the Proposition~\ref{prop:W-bound}. The estimate would remain valid with the factor $\abs{\Gamma(p)}$ rather than $\abs{\Gamma(\Mdemi+p-\nu)\Gamma(\Mdemi+p+\nu)}^{1/2}$ in the denominator of~\eqref{eq:W-bound:iii} as normalizing factor. Comparing with~\cite{HM06} this would be consistent with the weight zero case.

%%%%%%%%%%%%%%%%%%%%%%%%%%%%%%%
%                             %
%       Bibliographie         %
%                 %
%%%%%%%%%%%%%%%%%%%%%%%%%%%%%%%

%---avec amsrefs (Ordre a respecter)------
%\bibliography{TT}

\begin{bibsection}
\begin{biblist}

\bib{Bellman}{article}{
      author={Bellman, Richard},
       title={Ramanujan sums and the average value of arithmetic functions},
        date={1950},
        ISSN={0012-7094},
     journal={Duke Math. J.},
      volume={17},
       pages={159\ndash 168},
      review={\MR{MR0035312 (11,715h)}},
}

\bib{BH08}{article}{
      author={Blomer, Valentin},
      author={Harcos, Gergely},
       title={The spectral decomposition of shifted convolution sums},
        date={2008},
     journal={Duke Math. J.},
      volume={144},
      number={2},
       pages={321\ndash 339},
}

\bib{BH09}{article}{
      author={Blomer, Valentin},
      author={Harcos, Gergely},
       title={Twisted {$L$}-functions over number fields and {H}ilbert's
  eleventh problem},
        date={2010},
        ISSN={1016-443X},
     journal={Geom. Funct. Anal.},
      volume={20},
      number={1},
       pages={1\ndash 52},
         url={http://dx.doi.org/10.1007/s00039-010-0063-x},
      review={\MR{2647133 (2011g:11090)}},
}

\bib{BHM07}{article}{
      author={Blomer, V.},
      author={Harcos, G.},
      author={Michel, Ph.},
       title={A {B}urgess-like subconvex bound for twisted {$L$}-functions},
        date={2007},
        ISSN={0933-7741},
     journal={Forum Math.},
      volume={19},
      number={1},
       pages={61\ndash 105},
         url={http://dx.doi.org/10.1515/FORUM.2007.003},
        note={Appendix 2 by Z. Mao},
      review={\MR{2296066 (2008i:11067)}},
}

\bib{Blom08}{article}{
      author={Blomer, Valentin},
       title={Sums of {H}ecke eigenvalues over values of quadratic
  polynomials},
        date={2008},
        ISSN={1073-7928},
     journal={Int. Math. Res. Not. IMRN},
      number={16},
       pages={Art. ID rnn059. 29},
      review={\MR{2435749 (2009g:11053)}},
}

\bib{BM05}{article}{
      author={Bruggeman, Roelof~W.},
      author={Motohashi, Yoichi},
       title={A new approach to the spectral theory of the fourth moment of the
  {R}iemann zeta-function},
        date={2005},
     journal={J. Reine Angew. Math.},
      volume={579},
       pages={75\ndash 114},
}

\bib{BM10}{article}{
      author={Baruch, E.},
      author={Mao, Z.},
       title={A generalized {K}ohnen-{Z}agier formula for {M}aass forms},
        date={2010},
     journal={J. London Math. Soc.},
}

\bib{Bykovskii}{article}{
      author={Bykovski{\u\i}, V.~A.},
       title={Spectral expansions of certain automorphic functions and their
  number-theoretic applications},
        date={1984},
        ISSN={0373-2703},
     journal={Zap. Nauchn. Sem. Leningrad. Otdel. Mat. Inst. Steklov. (LOMI)},
      volume={134},
       pages={15\ndash 33},
        note={Automorphic functions and number theory, II},
      review={\MR{MR741852 (86g:11031)}},
}

\bib{DFI8}{article}{
      author={Duke, W.},
      author={Friedlander, J.},
      author={Iwaniec, H.},
       title={The subconvexity problem for {A}rtin \${L}\$-functions},
        date={2002},
        ISSN={0020-9910},
     journal={Invent. Math.},
      volume={149},
      number={3},
       pages={489\ndash 577},
}

\bib{DeI:82}{article}{
      author={Deshouillers, J.-M.},
      author={Iwaniec, H.},
       title={On the greatest prime factor of {$n^{2}+1$}},
        date={1982},
        ISSN={0373-0956},
     journal={Ann. Inst. Fourier (Grenoble)},
      volume={32},
      number={4},
       pages={1\ndash 11 (1983)},
         url={http://www.numdam.org/item?id=AIF_1982__32_4_1_0},
      review={\MR{MR694125 (84m:10033)}},
}

\bib{Duke88}{article}{
      author={Duke, W.},
       title={Hyperbolic distribution problems and half-integral weight {M}aass
  forms},
        date={1988},
        ISSN={0020-9910},
     journal={Invent. Math.},
      volume={92},
      number={1},
       pages={73\ndash 90},
}

\bib{book:Gelbart}{book}{
      author={Gelbart, Stephen~S.},
       title={Weil's representation and the spectrum of the metaplectic group},
      series={Lecture Notes in Mathematics, Vol. 530},
   publisher={Springer-Verlag},
     address={Berlin},
        date={1976},
      review={\MR{MR0424695 (54 \#12654)}},
}

\bib{Good83}{article}{
      author={Good, Anton},
       title={On various means involving the {F}ourier coefficients of cusp
  forms},
        date={1983},
        ISSN={0025-5874},
     journal={Math. Z.},
      volume={183},
      number={1},
       pages={95\ndash 129},
         url={http://dx.doi.org/10.1007/BF01187218},
      review={\MR{MR701361 (84g:10052)}},
}

\bib{GR}{book}{
      author={Gradshteyn, I.~S.},
      author={Ryzhik, I.~M.},
       title={Table of integrals, series, and products},
     edition={Seventh},
   publisher={Elsevier/Academic Press, Amsterdam},
        date={2007},
        ISBN={978-0-12-373637-6; 0-12-373637-4},
        note={Translated from the Russian, Translation edited and with a
  preface by Alan Jeffrey and Daniel Zwillinger, With one CD-ROM (Windows,
  Macintosh and UNIX)},
      review={\MR{MR2360010 (2008g:00005)}},
}

\bib{GS83}{article}{
      author={Goldfeld, D.},
      author={Sarnak, P.},
       title={Sums of {K}loosterman sums},
        date={1983},
        ISSN={0020-9910},
     journal={Invent. Math.},
      volume={71},
      number={2},
       pages={243\ndash 250},
      review={\MR{MR689644 (84e:10037)}},
}

\bib{Hansen}{article}{
      author={Hansen, D.},
       title={{M}ordell-{W}eil growth for {GL2}-type abelian varieties over
  {H}ilbert class fields of {CM} fields},
      eprint={http://arxiv.org/abs/1005.4700},
}

\bib{Harc03}{article}{
      author={Harcos, G.},
       title={An additive problem in the {F}ourier coefficients of cusp forms},
        date={2003},
        ISSN={0025-5831},
     journal={Math. Ann.},
      volume={326},
      number={2},
       pages={347\ndash 365},
}

\bib{HM06}{article}{
      author={Harcos, G.},
      author={Michel, Ph.},
       title={The subconvexity problem for {R}ankin-{S}elberg \${L}\$-functions
  and equidistribution of {H}eegner points. {II}},
        date={2006},
        ISSN={0020-9910},
     journal={Invent. Math.},
      volume={163},
      number={3},
       pages={581\ndash 655},
}

\bib{Hool63}{article}{
      author={Hooley, Christopher},
       title={On the number of divisors of a quadratic polynomial},
        date={1963},
        ISSN={0001-5962},
     journal={Acta Math.},
      volume={110},
       pages={97\ndash 114},
}

\bib{Hooley63}{article}{
      author={Hooley, Christopher},
       title={On the number of divisors of a quadratic polynomial},
        date={1963},
        ISSN={0001-5962},
     journal={Acta Math.},
      volume={110},
       pages={97\ndash 114},
      review={\MR{MR0153648 (27 \#3610)}},
}

\bib{book:IK04}{book}{
      author={Iwaniec, Henryk},
      author={Kowalski, Emmanuel},
       title={Analytic number theory},
      series={American Mathematical Society Colloquium Publications},
   publisher={American Mathematical Society},
     address={Providence, RI},
        date={2004},
      volume={53},
        ISBN={0-8218-3633-1},
}

\bib{Iwaniec}{book}{
      author={Iwaniec, Henryk},
       title={Spectral methods of automorphic forms},
     edition={Second},
      series={Graduate Studies in Mathematics},
   publisher={American Mathematical Society},
     address={Providence, RI},
        date={2002},
      volume={53},
        ISBN={0-8218-3160-7},
      review={\MR{MR1942691 (2003k:11085)}},
}

\bib{Kim-Sarnak}{article}{
      author={Kim, Henry~H.},
       title={Functoriality for the exterior square of {${\rm GL}\sb 4$} and
  the symmetric fourth of {${\rm GL}\sb 2$}},
        date={2003},
        ISSN={0894-0347},
     journal={J. Amer. Math. Soc.},
      volume={16},
      number={1},
       pages={139\ndash 183 (electronic)},
        note={With appendix 1 by Dinakar Ramakrishnan and appendix 2 by Kim and
  Peter Sarnak},
      review={\MR{MR1937203 (2003k:11083)}},
}

\bib{KS}{article}{
      author={Katok, Svetlana},
      author={Sarnak, Peter},
       title={Heegner points, cycles and {M}aass forms},
        date={1993},
        ISSN={0021-2172},
     journal={Israel J. Math.},
      volume={84},
      number={1-2},
       pages={193\ndash 227},
      review={\MR{MR1244668 (94h:11051)}},
}

\bib{book:Lang:sl2}{book}{
      author={Lang, Serge},
       title={{${\rm SL}_2({\bf R})$}},
      series={Graduate Texts in Mathematics},
   publisher={Springer-Verlag},
     address={New York},
        date={1975},
      volume={105},
        ISBN={0-387-96198-4},
        note={Reprint of the 1975 edition},
}

\bib{LRS}{article}{
      author={Luo, W.},
      author={Rudnick, Z.},
      author={Sarnak, P.},
       title={On {S}elberg's eigenvalue conjecture},
        date={1995},
        ISSN={1016-443X},
     journal={Geom. Funct. Anal.},
      volume={5},
      number={2},
       pages={387\ndash 401},
      review={\MR{MR1334872 (96h:11045)}},
}

\bib{Michel-park}{incollection}{
      author={Michel, Ph.},
       title={Analytic number theory and families of automorphic
  {L}-functions},
        date={2007},
   booktitle={Automorphic forms and applications},
      series={IAS/Park City Math. Ser.},
      volume={12},
   publisher={Amer. Math. Soc.},
     address={Providence, RI},
       pages={181\ndash 295},
      review={\MR{MR2331346}},
}

\bib{cong:park:mich}{incollection}{
      author={Michel, Ph.},
       title={Analytic number theory and families of automorphic
  \${L}\$-functions},
   booktitle={Automorphic forms and applications},
      series={IAS/Park City Math. Ser.},
      volume={12},
   publisher={Amer. Math. Soc.},
     address={Providence, RI},
       pages={181\ndash 295},
}

\bib{MV10}{article}{
      author={Michel, Philippe},
      author={Venkatesh, Akshay},
       title={The subconvexity problem for {${\rm GL}_2$}},
        date={2010},
        ISSN={0073-8301},
     journal={Publ. Math. Inst. Hautes \'Etudes Sci.},
      number={111},
       pages={171\ndash 271},
         url={http://dx.doi.org/10.1007/s10240-010-0025-8},
      review={\MR{2653249}},
}

\bib{Pras93}{incollection}{
      author={Prasad, Dipendra},
       title={Weil representation, {H}owe duality, and the theta
  correspondence},
        date={1993},
   booktitle={Theta functions: from the classical to the modern},
      series={CRM Proc. Lecture Notes},
      volume={1},
   publisher={Amer. Math. Soc.},
     address={Providence, RI},
       pages={105\ndash 127},
      review={\MR{MR1224052 (94e:11043)}},
}

\bib{PS84}{incollection}{
      author={Piatetski-Shapiro, Ilya},
       title={Work of {W}aldspurger},
        date={1984},
   booktitle={Lie group representations, {II} ({C}ollege {P}ark, {M}d.,
  1982/1983)},
      series={Lecture Notes in Math.},
      volume={1041},
   publisher={Springer},
     address={Berlin},
       pages={280\ndash 302},
      review={\MR{748511 (86g:11030)}},
}

\bib{Sarn84}{incollection}{
      author={Sarnak, Peter},
       title={Additive number theory and {M}aass forms},
        date={1984},
   booktitle={Number theory ({N}ew {Y}ork, 1982)},
      series={Lecture Notes in Math.},
      volume={1052},
   publisher={Springer},
     address={Berlin},
       pages={286\ndash 309},
      review={\MR{MR750670 (86f:11042)}},
}

\bib{Sarn94}{article}{
      author={Sarnak, Peter},
       title={Integrals of products of eigenfunctions},
        date={1994},
        ISSN={1073-7928},
     journal={Internat. Math. Res. Notices},
      number={6},
       pages={251 ff., approx. 10 pp. (electronic)},
}

\bib{Selb65}{incollection}{
      author={Selberg, Atle},
       title={On the estimation of {F}ourier coefficients of modular forms},
        date={1965},
   booktitle={Proc. {S}ympos. {P}ure {M}ath., {V}ol. {VIII}},
   publisher={Amer. Math. Soc.},
     address={Providence, R.I.},
       pages={1\ndash 15},
      review={\MR{MR0182610 (32 \#93)}},
}

\bib{Cohen-Sarnak}{article}{
      author={Selberg, A.},
       title={Notes on {S}elberg's lectures by {C}ohen and {S}arnak},
}

\bib{Shimura73}{article}{
      author={Shimura, Goro},
       title={On modular forms of half integral weight},
        date={1973},
        ISSN={0003-486X},
     journal={Ann. of Math. (2)},
      volume={97},
       pages={440\ndash 481},
      review={\MR{MR0332663 (48 \#10989)}},
}

\bib{Se-St}{incollection}{
      author={Serre, J.-P.},
      author={Stark, H.~M.},
       title={Modular forms of weight {$1/2$}},
        date={1977},
   booktitle={Modular functions of one variable, {VI} ({P}roc. {S}econd
  {I}nternat. {C}onf., {U}niv. {B}onn, {B}onn, 1976)},
   publisher={Springer},
     address={Berlin},
       pages={27\ndash 67. Lecture Notes in Math., Vol. 627},
      review={\MR{MR0472707 (57 \#12400)}},
}

\bib{Temp:non-split}{article}{
      author={Templier, N.},
       title={A non-split sum of coefficients of modular forms},
        date={2011},
     journal={Duke Math. J.},
      volume={157},
      number={1},
       pages={109\ndash 165},
}

\bib{Temp:height}{article}{
      author={Templier, N.},
       title={Minoration du rang des courbes elliptiques sur les corps de
  classes de {H}ilbert},
     journal={To appear in Compositio Math.},
}

\bib{Venk05}{article}{
      author={Venkatesh, Akshay},
       title={Sparse equidistribution problems, period bounds and
  subconvexity},
        date={2010},
        ISSN={0003-486X},
     journal={Ann. of Math. (2)},
      volume={172},
      number={2},
       pages={989\ndash 1094},
         url={http://dx.doi.org/10.4007/annals.2010.172.989},
      review={\MR{2680486}},
}

\bib{Wald80}{article}{
      author={Waldspurger, J.-L.},
       title={Correspondance de {S}himura},
        date={1980},
        ISSN={0021-7824},
     journal={J. Math. Pures Appl. (9)},
      volume={59},
      number={1},
       pages={1\ndash 132},
      review={\MR{MR577010 (83f:10029)}},
}

\bib{book:WW}{book}{
      author={Whittaker, E.~T.},
      author={Watson, G.~N.},
       title={A course of modern analysis},
      series={Cambridge Mathematical Library},
   publisher={Cambridge University Press},
     address={Cambridge},
        date={1996},
        ISBN={0-521-58807-3},
        note={An introduction to the general theory of infinite processes and
  of analytic functions; with an account of the principal transcendental
  functions, Reprint of the fourth (1927) edition},
      review={\MR{MR1424469 (97k:01072)}},
}


\end{biblist}
\end{bibsection}
% Recopier le fichier Main.bbl en un fichier .ltb

\end{document}